\documentclass{article}

\usepackage{PRIMEarxiv}
\usepackage[utf8]{inputenc} 
\usepackage[T1]{fontenc}    
\usepackage{hyperref}       
\usepackage{url}            
\usepackage{booktabs}       
\usepackage{amsfonts}       
\usepackage{nicefrac}       
\usepackage{microtype}      
\usepackage{lipsum}
\usepackage{fancyhdr}       
\usepackage{graphicx}       
\usepackage{amsmath}
\usepackage{amsthm}
\usepackage{amsfonts}
\usepackage{hyperref}
\usepackage{orcidlink}

\usepackage[numbers,sort&compress]{natbib}
\numberwithin{equation}{section}

\newtheorem{theorem}{Theorem}[section]
\newtheorem{protocol}[theorem]{Protocol}
\newtheorem{lemma}[theorem]{Lemma}
\newtheorem{corollary}{Corollary}[theorem]

\theoremstyle{definition}
\newtheorem{definition}{Definition}[section]

\title{FTNILO: Explicit Multivariate Function Inversion, Optimization and Counting, Cryptography Weakness and Riemann Hypothesis Solution equation with Tensor Networks
}
\author{
  Alejandro Mata Ali 
\orcidlink{https://orcid.org/0009-0006-7289-8827}\\
  Instituto Tecnológico de Castilla y León \\
  Burgos, Spain\\
  \texttt{alejandro.mata@itcl.es} \\
  }

\begin{document}

\maketitle
\begin{abstract}
In this paper, we present a new formalism, the Field Tensor Network Integral Logical Operator \mbox{(FTNILO)}, to obtain the explicit equation that returns the minimum, maximum, and zeros of a multivariable injective function, and an algorithm for non-injective ones. This method extends the MeLoCoToN algorithm for inversion and optimization problems with continuous variables, by using Field Tensor Networks. The fundamentals of the method are the conversion of the problem of minimization of $N$ continuous variables into a problem of maximization of a dependent functional of a single variable. It can also be adapted to determine other properties, such as the zeros of any function. For this purpose, we use an extension of the imaginary time evolution, the new method of continuous signals, and partial or total integration, depending on the case. In addition, we show a direct way to recover both the tensor networks and the MeLoCoToN from this formalism. We show some examples of application, such as the Riemann hypothesis resolution. We provide an explicit integral equation that gives the solution of the Riemann hypothesis, being that if it results in a zero value, it is correct; otherwise, it is wrong. This algorithm requires no deep mathematical knowledge and is based on simple mathematical properties.
\end{abstract}


\tableofcontents

\newpage

\section{Introduction}
Function optimization is a field of great current interest, both academic and applied. Many physical problems can be modeled as obtaining the configuration that corresponds to the minimum of a function or functional, for example, energy~\cite{Ground_State} or action~\cite{Landau_1975pou}. Many applied problems can also be modeled in this way, such as portfolio optimization~\cite{Portfolio}, the optimization of production manufacturing processes~\cite{Production_Managment}, machine learning training~\cite{ML_optim}, or energy efficiency~\cite{Energy_efficiency}. This can be seen as a generalization of combinatorial optimization problems to continuous variables.

In the case of a function $f(x)$ with a single variable $x$, it is easy to obtain the minimum $X$, simply by applying
\begin{equation}
    \left.\frac{df(x)}{dx}\right|_{x=X} = 0, \quad \left.\frac{d^2f(x)}{dx^2}\right|_{x=X} > 0,
\end{equation}
and solving the resulting equations. Analogously to determine the maximum. The difficulty of computing the derivative of the function, the non-differentiable case, or solving the resulting equations from the differentiation can be hard problems in the search of the minimum from this expression. Also, we need to determine which of all the local minimum points is the global minimum, which can be difficult if there is an infinite number of them.

The case of functions with a larger number of variables $f(x_0,x_1,\dots,x_{N-1})$ is much more complex, since the gradient of the function must be calculated and evaluated at the points where its value equals zero
\begin{equation}
    \left.\nabla f(x_0,x_1,\dots,x_{N-1})\right|_{X_0,X_1,\dots,X_{N-1}} = \vec{0}.
\end{equation}
Those are the critical points. After that, we should compute the Hessian Matrix
\begin{equation}
    \mathcal{H}[f(x_0,x_1,\dots,x_{N-1})]_{ij} = \frac{\partial^2 f(x_0,x_1,\dots,x_{N-1})}{\partial x_i\partial x_j},
\end{equation}
with which to determine which of the points are minima knowing that they are those at which the Hessian is positive definite. In this case, we have the same problems as in the one-variable case but with many more computations. For this reason, there exists a wide literature of methods to solve optimization problems, such as stochastic gradient descent~\cite{Stochastic}, Quasi-Newton Methods~\cite{ConditioningOQ}, or derivative-free~\cite{Gradient_Free} and Hessian-free optimization~\cite{Hessian_Free}.

Another relevant problem is obtaining concrete values of functions. That is, the continuous generalization of combinatorial inversion problems. Given a known function $f(x)$, we want to know the value $X$ such that $f(X)=0$ is satisfied. In the multivariable case, we want to obtain the vector $\vec{X}$ such that $f(\vec{X})=0$. Obviously, this problem is related to the previous one. In this case, the difficulty lies directly in finding these points. There are several search methods, such as the bisection method~\cite{bisection}, the Newton-Raphson method~\cite{Newton-Raphson}, or the secant method~\cite{Secant}, but they have associated difficulties. A case of extreme interest in the search for zeros is the Riemannian Zeta function~\cite{Riemann_Funct}, defined as
\begin{equation}
    \zeta (s)=\sum_{n=1}^{\infty} \frac{1}{n^s},\ s\in \mathbb{C},
\end{equation}
for $Re(s)>1$, and extended for other values of $s$. This problem is related to the Riemann hypothesis~\cite{Riemann_Hypo}, which states that all non-trivial zeros of the Riemann Zeta function follow $Re(s)=1/2$. This is one of the unsolved Millennium Prize Problems~\cite{Millenium_Grand}. However, it has more applications, such as artificial vision~\cite{inversion_vision,perception_inverse}, simulation, prediction~\cite{Inverse_Ocean}, or resolution of other mathematical problems, and there are several methods to approach them~\cite{Inverse_modeling,Optical_inverse,Reservoir_inverse}.

Quantum computing has recently gained great popularity because of its ability to tackle various problems more efficiently than known classical algorithms. The best known example is Shor's algorithm~\cite{Shor}. However, given the limitation of current hardware, in the Noisy Intermediate-Scale Quantum (NISQ) era~\cite{NISQ}, many of these algorithms are not feasible. For example, Shor's algorithm requires 20 million qubits to break RSA-2048~\cite{shor20M}, something that today is impossible. This has led to a growing interest in an alternative field, that of quantum inspiration and, more particularly, the tensor networks. Tensor networks~\cite{TN_Orus} are graphical representations of certain tensor computations, which seek to exploit some quantum mathematical properties in classical devices. Among their use cases, the most famous is to efficiently simulate low-entanglement quantum systems~\cite{TN_Simul}. It has also been explored for machine learning models~\cite{TNML} and for combinatorial optimization~\cite{TN_Sime}. Within combinatorial optimization, we highlight the MeLoCoToN algorithm~\cite{Melocoton} that we take as a starting point for the new development, which allows us to obtain the explicit and exact formula that solves any combinatorial problem, both optimization and inversion.

In this context, we present the Field Tensor Network Integral Logical Operator (\mbox{FTNILO}). This is a tensor network formalism to obtain the explicit equation that returns the global minimum or global maximum of a multivariate function composed of other functions. The key of the algorithm is the expression of the function as a logical circuit of operator vectorial functions, and the integration over all of them. This follows and generalizes the MeLoCoToN algorithm, transforming the discrete signals into continuous signals and the summations into integrals. This formalism also allows one to create an explicit equation that returns other properties of the function, such as the location of specific values or the number of zeros in one input region. We use this formalism to recover the MeLoCoToN one as a particular case and to make some pure mathematics statements. We show its potential by obtaining the equation that gives the number of zeros of the Riemann Zeta function for the $Re(s)>0$ region. With this equation, we also obtain an equation whose value returns the demonstration (or not) of the Riemann hypothesis.

\newpage
\section{Field Tensor Network Integral Logical Operator (FTNILO)}
This formalism is based on the discrete case for combinatorial problems shown in~\cite{Melocoton}, where the MeLoCoToN method is defined. In order to focus this paper on the new method, we will not describe the MeLoCoToN and the concepts presented there, since it can be consulted openly. We will start by presenting the Field Tensor Networks (FTN). This type of generalized tensor network was first presented in the paper~\cite{Field_TN}. Since its notation and formalism are specialized in its concrete case of study, we are going to make modifications to better adapt it to our problems.

\begin{figure}[h]
    \centering
    \includegraphics[width=0.4\linewidth]{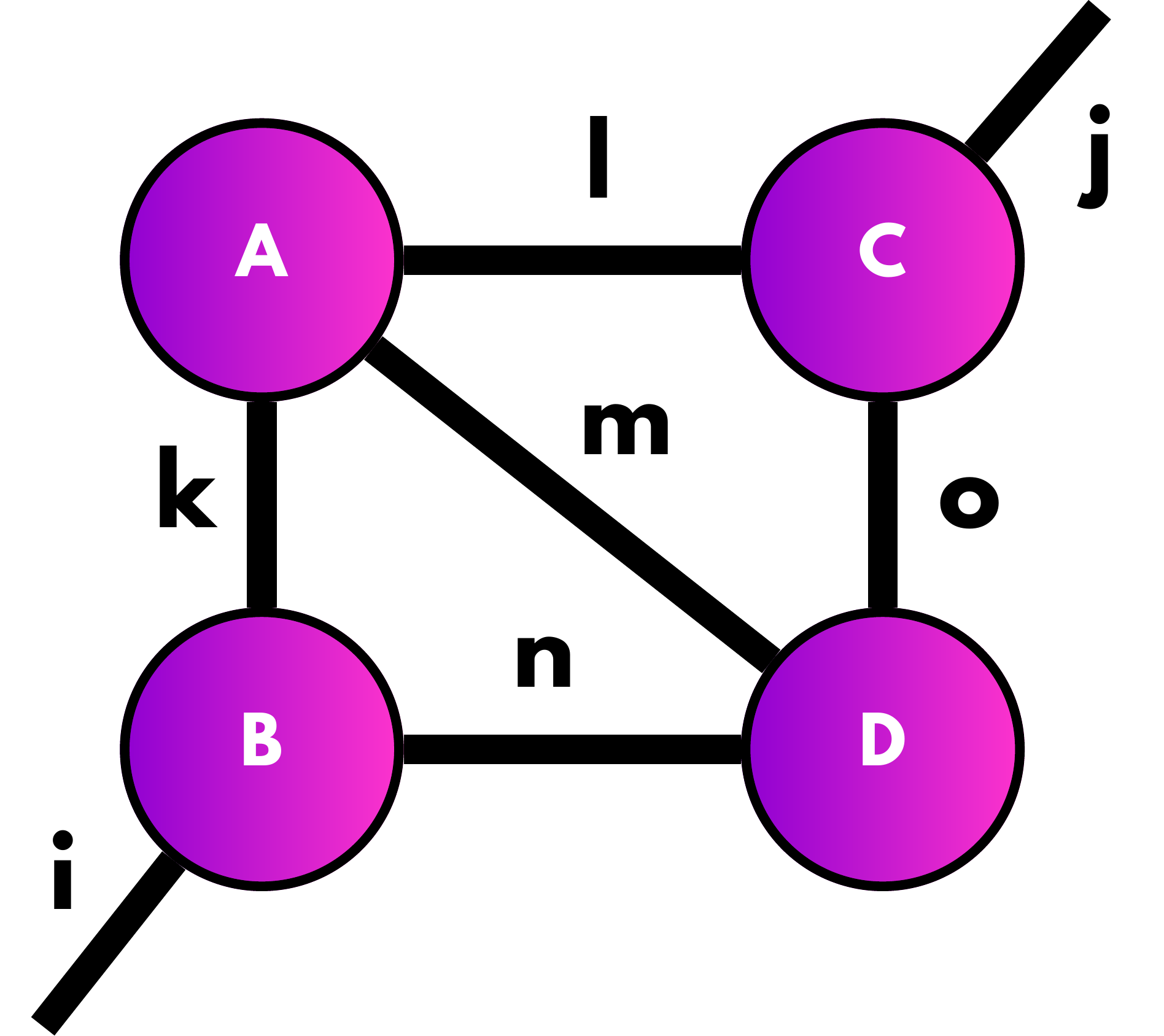}
    \caption{Tensor Network with four tensors.}
    \label{fig: TN basic}
\end{figure}

As we know, a tensor network represents a summation of products of tensor elements. For example, the tensor network shown in Fig.~\ref{fig: TN basic} represents the equation
\begin{equation}
    E_{ij} = \sum_{klmno} A_{klm} B_{ikn} C_{ilo}D_{nmo},
\end{equation}
being $i,j,k,l,m,n,o$ natural numbers and $A,B,C,D,E$ tensors. We say that $i,j,k,l,m,n,o$ are the indexes, $i,j$ are free and $k,l,m,n,o$ are bonded.

A field tensor network represents a multidimensional integral of products of functions. In this paper, we would always use generalized functions like Dirac deltas.
The FTN equation that would be associated with Fig.~\ref{fig: TN basic} representation is
\begin{equation}
    E(i,j) = \int A(k,l,m) B(i,k,n) C(i,l,o)D(n,m,o)dk dl dm dn do,
\end{equation}
being $i,j,k,l,m,n,o$ real numbers and $A,B,C,D,E$ multivariate functions, which we call \textit{tensor functions}. We say that $i,j,k,l,m,n,o$ are the \textit{continuous indexes}, $i,j$ are free and $k,l,m,n,o$ are bonded. The analogy between the two equations is clear. The only change is the substitution of each tensor of elements $T_{xyz}$ into a function $T(x,y,z)$ and the summation into an integral. The integration region is the equivalence of the dimension of the index, but in general we will define it considering that it can be all the real space. In the original paper, the operators are functionals and the indexes are functions, but in our approach it is more convenient to use generalized functions directly instead of functionals.

With this simple generalization, all the work developed in the MeLoCoToN formalism is easily generalizable to the field tensor networks formalism. We will gradually develop the new formalism for both multivariate function optimization and function inversion. We will understand every step with some examples. First, we will consider the case with only one solution, with no degeneration. After that, we will consider the degenerate case.

\subsection{Logical Circuits}
As in MeLoCoToN, we have to start by properly defining the problem variables and the classical logical circuit that does what we need. The choice of problem variables follows the same philosophy as in the discrete method: each operator must receive the minimum amount of information necessary to operate. With this, we can make the equations really express the internal relationships of the problem, and they are simpler to compute. In this case, we have the input variables of the target function, and the internal signal variables, which are the ones we can choose.

The logical circuit receives an input $\vec{x}$ and returns an output $\vec{y} = \gamma(\vec{x})$ (which can be the same), associating an internal number with the circuit given by the input, called the \textit{amplitude} in analogy to quantum systems. Since the circuit will be converted to a tensor network, each operator can only multiply the total amplitude of the circuit by one number.

\subsubsection{Inversion Problem}
In the case of inverting vector-valued functions, the logical circuit receives $\vec{x}$ and outputs $\vec{y} = f(\vec{x})$. Formally, $f:\mathcal{X}\rightarrow \mathcal{Y}$, with $\mathcal{X}\subseteq  \mathbb{R}^n$ and $\mathcal{Y}\subseteq \mathbb{R}^m$. This is performed using a series of \textit{transformation operators}. Each logical operator receives the corresponding part of the input and transforms it. Given the correlated transformations, it also communicates to the other operators signals with relevant information to perform their computations properly. It is the same as in the MeLoCoToN, but with continuous inputs and signals.

\paragraph{Computing the zeros from the sum of a sequence}

$ $

Given a set of real numbers $\{a_k\in \mathbb{R}, \forall k\in [0,n-1]\}$ and a value $Y\in\mathbb{R}$, we want the $\vec{X}\in \mathbb{R}^n$ which satisfies $Y=f(\vec{X})$, being the $f$ function defined as
\begin{equation}\label{eq: first inversion function}
    f(\vec{x})=\sum_{k=0}^{n-1} (a_k)^{x_k}.
\end{equation}
If $Y=0$, we are computing the zeros of the function.

We need a logical circuit that receives the values of each component of the $\vec{x}$ vector and returns the value of $f(\vec{x})$. The simplest way is to create a set of operators that compute each $(a_k)^{x_k}$ and sum them iteratively in a chain. The circuit is shown in Fig.~\ref{fig: Logical circuits} a. The $k$-th operator $S^k$ receives from its left input the value of $x_k$ and by its upper input the value of the sum up to the $k$-th step, $r_k=\sum_{j=0}^{k-1} (a_j)^{x_j}$. It returns by its lower output the sum received plus its own term $(a_k)^{x_k}$, returning $r_{k+1}=r_k+(a_k)^{x_k}=\sum_{j=0}^{k} (a_j)^{x_j}$. None of the operators modify the amplitude of the circuit.

\begin{figure}[h]
    \centering
    \includegraphics[width=\linewidth]{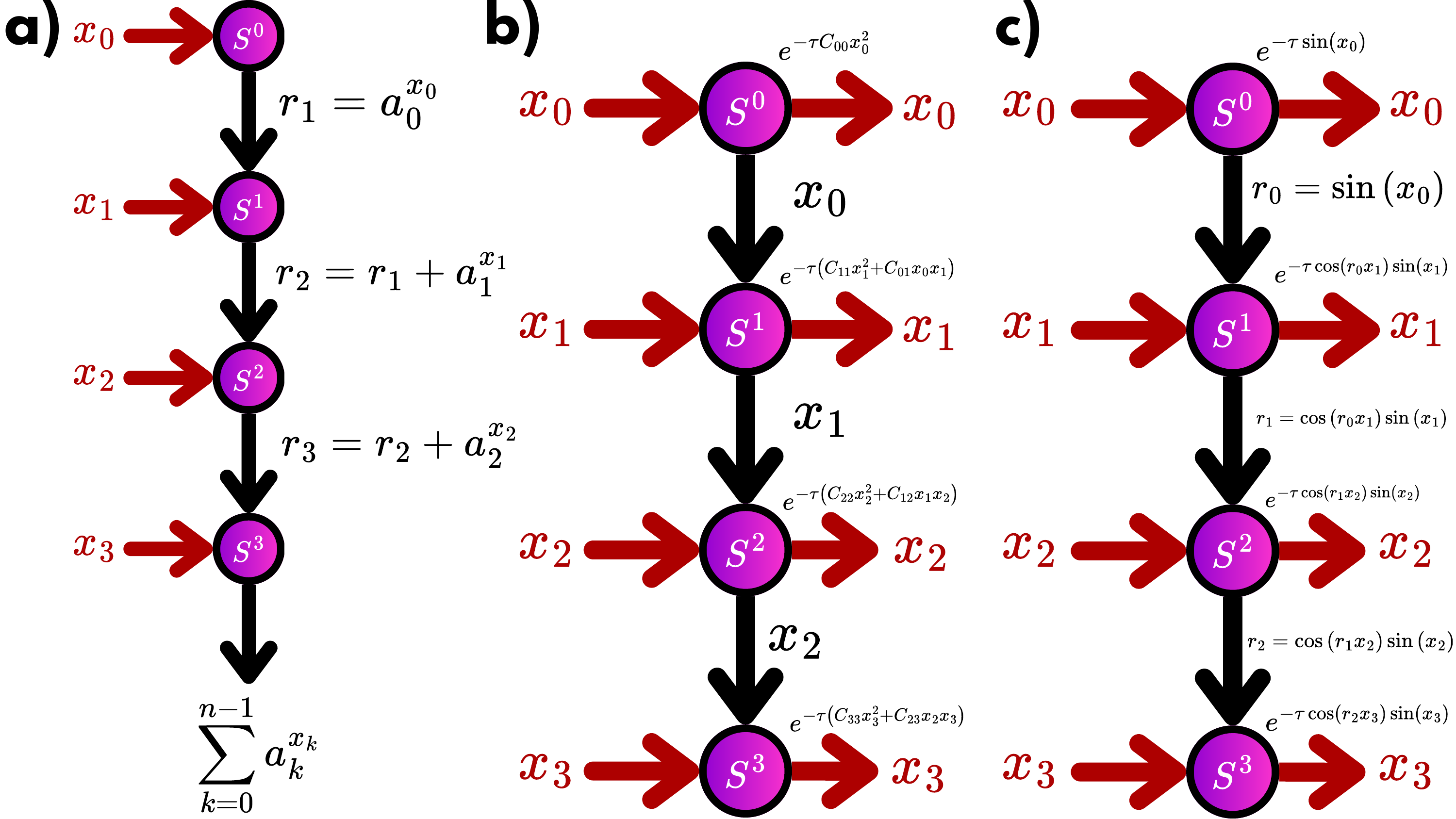}
    \caption{Logical circuits for a) Computing the zeros from the sum of a sequence, b) Optimizing a quadratic function with linear chain interaction to one neighbor, c) Optimizing the cos-sin function.}
    \label{fig: Logical circuits}
\end{figure}

\subsubsection{Optimization Problem}
In the case of optimizing a function $f:\mathcal{X}\rightarrow\mathcal{Y}$ with $\mathcal{X}\subseteq \mathbb{R}^n$ and $\mathcal{Y}\subseteq \mathbb{R}$, this is a logical circuit that receives as input the $\vec{x}$ value of the arguments of the function $f(\vec{x})$ and returns as output the same values, but by modifying the \textit{amplitude}. Its amplitude value is $e^{-\tau f(\vec{x})}$, for the \textit{imaginary time evolution}, being $\tau$ the imaginary time constant. Each logical operator has the ability to multiply the total amplitude of the circuit by a value, depending on the inputs it receives from other operators. In addition, each operator sends the other operators a signal, which carries information on how they should operate. This is exactly the same as in the MeLoCoToN case, but this time with continuous signals. In the case of functions that can be expressed as a sum of terms
\begin{equation}
    f(\vec{x})=\sum_i f_i (\rho_i(\vec{x})),
\end{equation}
being $\rho_i(\vec{x})$ the $i$-th subset of the input variables $\vec{x}$, the exponential follows
\begin{equation}
    e^{-\tau f(\vec{x})} = \prod_i e^{-\tau f_i (\rho_i(\vec{x}))},
\end{equation}
allowing its application with operators which receive less information of the input.

Additionally, if the optimal solution must follow a series of restrictions, preventing it from having a series of values, there will be other logical operators that, together with those of the optimizing circuit, will be in charge of multiplying by zero the amplitude of the inputs that do not satisfy the restrictions.

\paragraph{Optimizing a quadratic function with linear chain interaction to one neighbor}

$ $

Given a matrix $C$, we want to determine the minimum of the quadratic function
\begin{equation}
    f(\vec{x})=\sum_{i,j=0}^{n-1}C_{ij}x_ix_j.
\end{equation}
To simplify, we choose a version with linear chain interaction to one neighbor
\begin{equation}\label{eq: linear chain cost}
    f(\vec{x})=\sum_{i=0}^{n-1}(C_{ii}x_i^2+C_{i,i+1}x_ix_{i+1}).
\end{equation}
To obtain the minimum, we want the circuit to receive each component of $\vec{x}$ and change the amplitude of the circuit to $e^{-\tau f(\vec{x})}$. Considering that
\begin{equation}
    \begin{gathered}
    e^{-\tau f(\vec{x})} = e^{-\tau \sum_{i=0}^{n-1}(C_{ii}x_i^2+C_{i,i+1}x_ix_{i+1})}=\\
    = \prod_{i=0}^{n-1} e^{-\tau C_{ii}x_i^2}e^{-\tau C_{i,i+1}x_ix_{i+1}}.
    \end{gathered}
\end{equation}
In this case, the function is the linear combination of individual terms, so we can apply each term evolution using an operator. Moreover, each term depends only on the previous and current variables, so the logical circuit can be expressed as a chain. The logical circuit is shown in Fig.~\ref{fig: Logical circuits} b. The $i$-th operator $S^i$ receives by its left input the value of $x_i$ and by its upper input the value of the previous variable $x_{i-1}$. With this information, it can perform the multiplication of the amplitude by a factor $e^{-\tau (C_{ii}x_i^2+C_{i-1,i}x_{i-1}x_{i})}$. It returns by its lower and right outputs the value of its variable $x_i$.

\paragraph{Optimizing the cos-sin function}

$ $

In this case, we want to optimize the function
\begin{equation}
\begin{gathered}
    f(\vec{x})=\sum_{i=0}^{n-1} a_i(x_{i-1},x_i,a_{i-1}) \sin(x_i),\\
    a_0=1,\ a_i(x_{i-1},x_i,a_{i-1}) = \cos(a_{i-1} \sin(x_{i-1})x_i).
\end{gathered}
\end{equation}
This is a hard function to optimize. We follow the same structure as in the previous case, but we change the action and output of each operator. The logical circuit is shown in Fig.~\ref{fig: Logical circuits} c. Now, the $i$-th operator $S^i$ receives by its left input the value of $x_i$ and by its upper input the auxiliary value $r_{i-1}=a_{i-1} \sin(x_{i-1})$. With this information, it can perform the multiplication of the amplitude by a factor $e^{-\tau\cos(r_{i-1}x_i)\sin(x_i)}$. It returns by its lower output the auxiliary value $r_{i}=\cos(r_{i-1}x_{i}) \sin(x_{i})$ and by its right output the value of its variable $x_i$. With this construction, each operator performs its corresponding part of the evolution and sends to the next the minimal information needed.

\paragraph{Optimizing a quadratic function with linear chain interaction to one neighbor with the extra restriction $\sum_{i=0}^{n-1}a_ix_i<W$.}

$ $

In this case, we want to optimize the cost function in Eq.~\ref{eq: linear chain cost}, but with an additional restriction
\begin{equation}
    \sum_{i=0}^{n-1}a_ix_i<W.
\end{equation}

\begin{figure}
    \centering
    \includegraphics[width=0.7\linewidth]{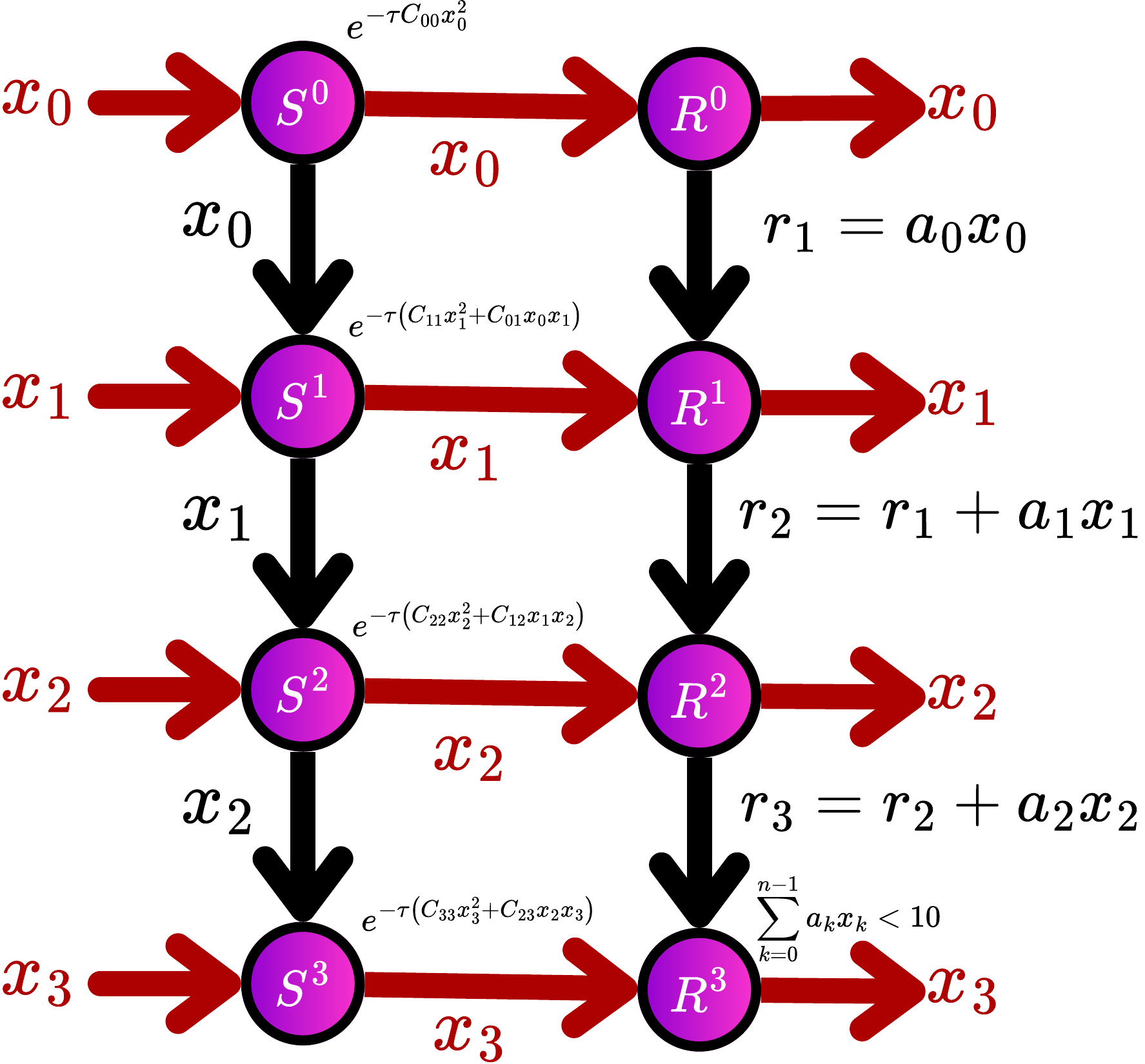}
    \caption{Logical Circuit for Optimizing a quadratic function with linear chain interaction to one neighbor with the restriction $\sum_k^{n-1}a_kx_k<10$.}
    \label{fig: Optim Restr}
\end{figure}

To approach this problem, we can reuse the logical circuit of the quadratic function case and concatenate it with other logical circuits to impose the restriction. The total circuit is shown in Fig.~\ref{fig: Optim Restr}. The second layer of the circuit is composed by $R$ operators, which evaluate the sum of the restriction and multiply the amplitude by zero if it does not satisfy it. The $i$-th operator $R^i$ receives as input the $i$-th variable value $x_i$ and the partial sum up to the $i$-th step, $r_i=\sum_{i=0}^{i-1}a_jx_j$. It returns its variable value $x_i$ to the right and the partial sum $r_{i+1}=r_i+a_ix_i = \sum_{i=0}^{i}a_jx_j$ to the bottom. Each operator multiplies the amplitude by one. The last $R$ operator verifies if the sum is lower than $W$. If it is, it multiplies the amplitude by one; if not, it multiplies the amplitude by zero. This zero multiplication can be interpreted as erasing the output due to the incompatible input.

\subsection{Circuit Field Tensorization}
After creating the logical circuit, we need to transform it into a field tensor network. This transformation allows us to process all possible inputs at the same time. The process is similar to the Circuit Tensorization in the MeLoCoToN. In general, each operator $A$ which receives the inputs $\vec{x}$, returns the outputs $\vec{y}=g(\vec{x})$ and multiplies the amplitude of the circuit by $h(\vec{x})$ is transformed to the generalized function
\begin{equation}\label{eq: function tensorized normalized}
A(\vec{x},\vec{\alpha})=
\begin{cases}
    h(\vec{x})& \text{ if }\vec{\alpha}=g(\vec{x}),\\
    0&\text{ otherwise }
\end{cases}
\end{equation}
With this transformation, the logical circuit is transformed directly into a field tensor network.

In the MeLoCoToN formalism, this operation is performed with \textit{Kronecker deltas}
\begin{equation}
    \delta_{ij\dots k} = 
    \begin{cases}
        1 &\text{ if } i=j=\dots=k,\\
        0 &\text{ otherwise } 
    \end{cases}
\end{equation}
being the resulting tensorization
\begin{equation}
    A_{\vec{x},\vec{\alpha}}=h(\vec{x})\prod_i\delta_{\alpha_i, (g(\vec{x}))_i}.
\end{equation}
In the FTNILO, these operators are performed with \textit{Dirac deltas}~\cite{Dirac}
\begin{equation}
\begin{gathered}
    \int_{\mathbb{R}^n} \delta^{(n)}(x_0-X_0,x_1-X_1,\dots,x_{n-1}-X_{n-1})f(x_0,x_1,\dots,x_{n-1}) dx_0dx_1\cdots dx_{n-1}=\\
    =f(X_0,X_1,\dots,X_{n-1}),
\end{gathered}
\end{equation}
which can be constructed as the product of individual deltas
\begin{equation}
    \delta^{(n)}(x_0-X_0,x_1-X_1,\dots,x_{n-1}-X_{n-1}) = \delta(x_0-X_0)\delta(x_1-X_1)\cdots\delta(x_{n-1}-X_{n-1}).
\end{equation}
To condense notation, we define
\begin{equation}
    \delta(\vec{x}-\vec{X}) \equiv \delta^{(n)}(x_0-X_0,x_1-X_1,\dots,x_{n-1}-X_{n-1}).
\end{equation}
In this way, the resulting tensor function is
\begin{equation}\label{eq: function tensorization deltas}
    A(\vec{x},\vec{\alpha})=h(\vec{x})\prod_i\delta\left(\alpha_i - (g(\vec{x}))_i\right)=h(\vec{x})\delta\left(\vec{\alpha} - g(\vec{x})\right).
\end{equation}
It is important to note that \eqref{eq: function tensorization deltas} is not equivalent to \eqref{eq: function tensorized normalized}, due to the divergence of the deltas at the desired point. However, this is exactly the main point, because this allows the general expression to retain the correct value of $h(\vec{x})$ after integrating. So, the correct expression is the delta one, while the cases one is only useful for understanding the function we want to implement. In the following, when we refer to the function defined in cases, we will neglect the infinity in the non-zero points.

It is clear that if we multiply the tensor functions associated with the logical operators of the circuit, the action of the Dirac deltas will allow us to retain only the values of the free variables that are consistent with all of the bond variables integrated. This is because if we had two deltas that have their non-zero points in different places, without overlapping, their product would always be zero. So, the expression of the general represented generalized function is zero at all the points which do not follow the logics we imposed.

\subsubsection{Inversion Problem}

In the inversion case, the function represented by the FTN is
\begin{equation}\label{eq: inversion function}
    \Phi(\vec{x},\vec{y})=
    \begin{cases}
        1&\text{ if } \vec{y}=f(\vec{x}),\\
        0&\text{ if } \vec{y}\neq f(\vec{x})
    \end{cases}
\end{equation}
or expressed with deltas
\begin{equation}\label{eq: inversion function with deltas}
    \Phi(\vec{x},\vec{y})=\delta(\vec{y}-f(\vec{x})).
\end{equation}
We call it the \textit{free inversion density}.

We need to impose the desired result $\vec{Y}$ of the function $f(\vec{X})$, so we connect to the general output continuous indexes a function which projects the solution into the value $\vec{Y}$. The function selected is the product of Dirac delta functions $\delta$
\begin{equation}
    \Delta(\vec{y},\vec{Y})=\prod_{i=0}^{m-1}\delta(y_i-Y_i),
\end{equation}
due to their property
\begin{equation}
    \int_{-\infty}^{\infty} g(y)\delta(y-Y) dy = g(Y).
\end{equation}
It is implemented by connecting each delta to its corresponding free output index. This gives us the result
\begin{equation}
    \mathcal{F}(\vec{x}) = \int_{\mathcal{Y}}\Phi(\vec{x},\vec{y}) \Delta(\vec{y},\vec{Y}) d\vec{y}
    = \int_{\mathcal{Y}}\delta(\vec{y}-f(\vec{x})) \prod_{i=0}^{m-1}\delta(y_i-Y_i) d\vec{y} 
    = \delta(\vec{Y}-f(\vec{x})).
\end{equation}
This is equivalent to the projection of the tensor into the correct subspace presented in the MeLoCoToN. If $\vec{Y}=f(\vec{X})$, the FTN represents the function
\begin{equation}\label{eq: final inversion tensor without deltas}
\mathcal{F}(\vec{x})=
\begin{cases}
    1&\text{ if } \vec{x}=\vec{X},\\
    0&\text{ if } \vec{x}\neq\vec{X}
\end{cases}
\end{equation}
or with the deltas
\begin{equation}\label{eq: final inversion tensor}
    \mathcal{F}(\vec{x})=\delta(\vec{x}-\vec{X}).
\end{equation}
We call it the \textit{restricted inversion density}.

The correctness of the approach is trivial. Notice that this function, as the free inversion density function, is the FTN we have constructed, so really we are not computing these integrals now. We have not computed explicit forms of these functions. We are obtaining the global FTN that represents these desired functions, and we will do the required computations at the end, taking advantage of commutations and properties of the integrals. The example for computing the zeros from the sum of a sequence case is shown in Fig.~\ref{fig: TLCs} a.

\subsubsection{Optimization Problem}

In the optimization case, the function represented by the FTN is
\begin{equation}
\Phi(\vec{x},\vec{y},\tau)=
\begin{cases}
        e^{-\tau f(\vec{x})}&\text{ if } \vec{y}=\vec{x} \text{ and } \vec{x}\in \mathcal{R},\\
        0 & \text{ otherwise }
\end{cases}
\end{equation}
being $\mathcal{R}$ the region of inputs that satisfies the restrictions. With deltas it is
\begin{equation}
    \Phi(\vec{x},\vec{y},\tau) = e^{-\tau f(\vec{x})} \delta(\vec{x}-\vec{y})\chi_{\mathcal{R}}(\vec{x}),
\end{equation}
being $\chi_{\mathcal{R}}(\vec{x})$ the \textit{indicator function} of the subset $\mathcal{R}$
\begin{equation}
    \chi_{\mathcal{R}}(\vec{x}) = \begin{cases}
    1 & \text{if } \vec{x}\in \mathcal{R},\\
    0  & \text{if } \vec{x}\notin \mathcal{R}.
    \end{cases}
\end{equation}
We call the $\Phi$ function the \textit{free optimization density}. Due to the fact that we will integrate over this variable, if the subset $\mathcal{R}$ is a discrete topological space, we will change the $\chi_{\mathcal{R}}(\vec{x})$ for $\sum_i \delta(\vec{x}-\vec{R}_i)$, being $\vec{R}_i$ the $i$-th element of the subset. This preserves the integral correct value in case that the correct intervals are infinitesimally small. It is important to note that this change is automatically done in the FTN of the circuit. We must change it in the computations we make to prove the demonstrations, not in the deduction of the real equations.

As we can observe, the free optimization density function is `diagonal', so we can remove the second set of arguments connecting each continuous index of output with a \textit{One constant function} $+(y)=1$. This results in
\begin{equation}
    \mathcal{F}(\vec{x},\tau) = \int_{\mathcal{X}}\Phi(\vec{x},\vec{y},\tau) d\vec{y}= \int_{\mathcal{X}} e^{-\tau f(\vec{x})} \delta(\vec{x}-\vec{y})\chi_{\mathcal{R}}(\vec{x}) d\vec{y} = e^{-\tau f(\vec{x})}\chi_{\mathcal{R}}(\vec{x}).
\end{equation}

In this way, the resulting function is
\begin{equation}
 \mathcal{F}(\vec{x},\tau)=
\begin{cases}
    e^{-\tau f(\vec{x})}&\text{ if } \vec{x}\in \mathcal{R},\\
    0& \text{ if } \vec{x}\notin \mathcal{R},
\end{cases}
\end{equation}
or in one line
\begin{equation}
     \mathcal{F}(\vec{x},\tau)=e^{-\tau f(\vec{x})}\chi_{\mathcal{R}}(\vec{x}).
\end{equation}
We call it the \textit{restricted optimization density}.

Fig.~\ref{fig: TLCs} b and c show the FTN of optimizing a quadratic function with linear chain interaction to one neighbor case and optimizing a quadratic function with linear chain interaction to one neighbor with the restriction $\sum_k^{n-1}a_kx_k<W$, respectively.
\begin{figure}
    \centering
    \includegraphics[width=0.7\linewidth]{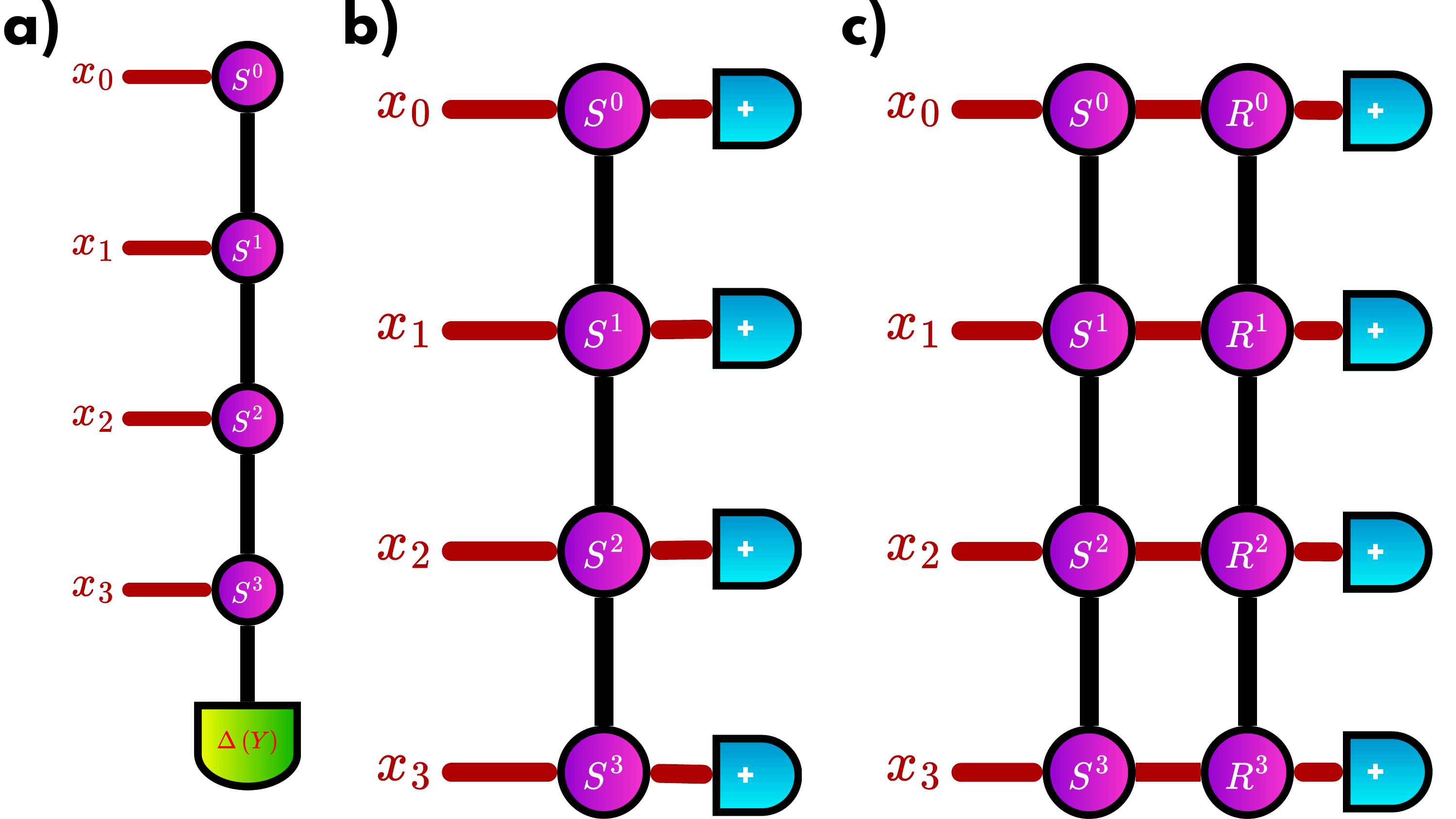}
    \caption{Tensorized circuits for a) Computing the zeros from the sum of a sequence, b) Optimizing a quadratic function with linear chain interaction to one neighbor, c) Optimizing a quadratic function with linear chain interaction to one neighbor with the restriction $\sum_k^{n-1}a_kx_k<10$.}
    \label{fig: TLCs}
\end{figure}

\subsection{Solution extraction}
Now that we have the tensor function that contains the important information of the problem, we need to extract the solution from it. In both cases, this process can be done iteratively, making use of tensor functions connected with the free indexes of the tensor function and updating the FTN for each step depending on the variable we want to determine.

\subsubsection{Inversion Problem}
In the inversion case, we can search for the solution in the large restricted inverse density function obtained $\mathcal{F}(\vec{x})$. However, this should be improved in order to obtain an explicit equation. To avoid brute-force search, in the determination of the $j$-th variable value, we can connect One Constant functions to each free index of the FTN but the $j$-th one, similar to the Half Partial Trace of the MeLoCoToN. This results in
\begin{equation}
    F_j(x_j)=\int_{\mathcal{X}_{-j}}\mathcal{F}(\vec{x}) \prod_{i=0\backslash i\neq j}^{n-1}dx_i =\int_{\mathcal{X}_{-j}} \prod_{i=0\backslash i\neq j}^{n-1}\delta(x_i-X_i)\prod_{i=1}^{n-1}dx_i =
    \delta(x_j-X_j),
\end{equation}
being $\mathcal{X}_{-j}$ the space of $\mathcal{X}$ neglecting the $j$-th dimension. We call this generalized function the \textit{j-th partial restricted inversion function}.

Because of this, we will only have to locate the position of the non-zero value of the delta. To do this, we can simply connect a node that implements the \textit{linear function} $L(x)=x$, as in Fig.~\ref{fig: Iteration inversion}, so that we have
\begin{equation}
    \Omega_j = \int_{\mathcal{X}_{j}} x_jF_j(x_j)dx_j = \int_{\mathcal{X}_{j}} x_j\delta(x_j-X_j)dx_j = X_j.
\end{equation}
We call this value the \textit{j-th partial restricted inversion value}.

\begin{figure}
    \centering
    \includegraphics[width=0.7\linewidth]{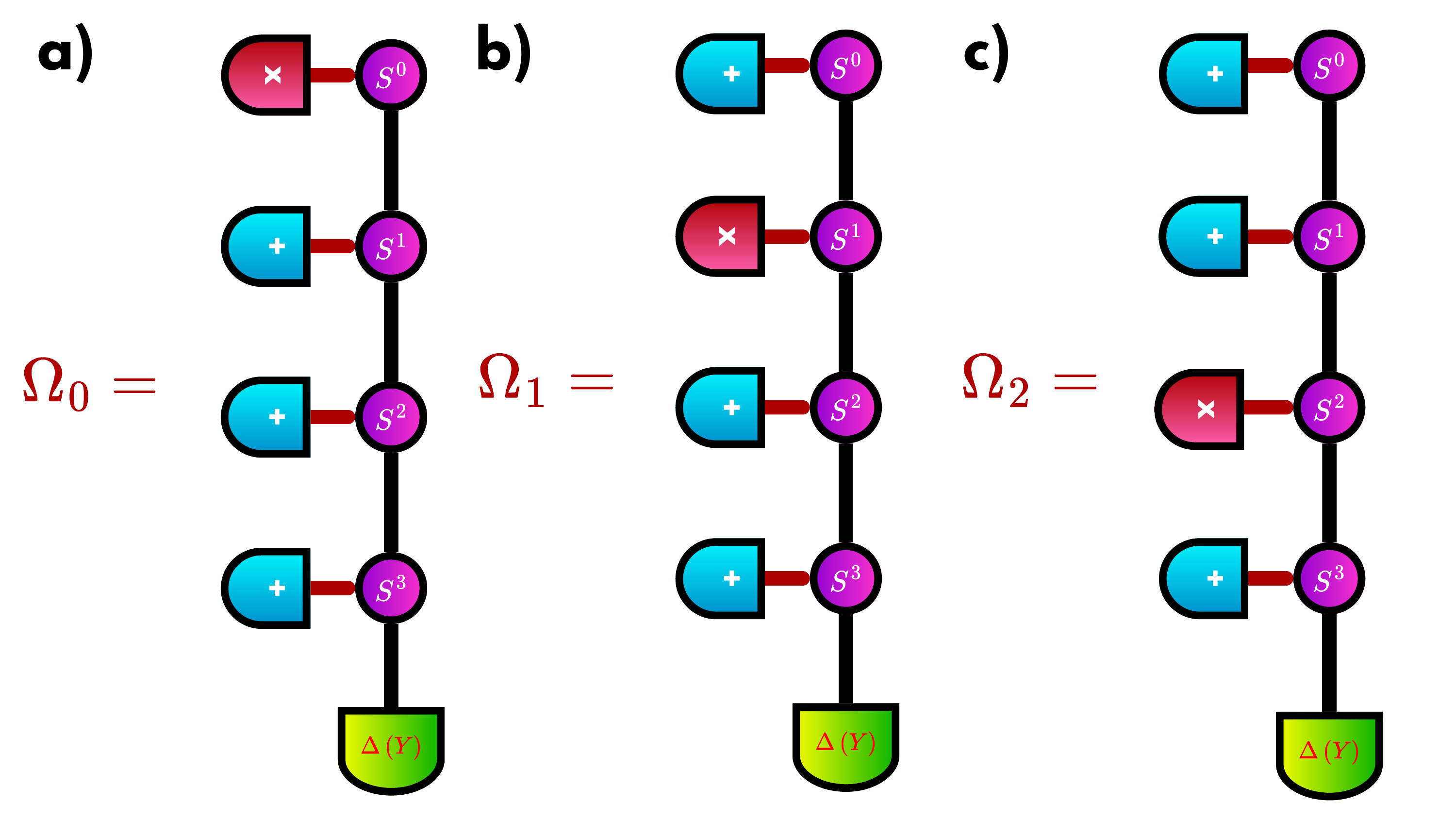}
    \caption{Iteration process to determine first, second and third variable values in the inversion problem for \eqref{eq: first inversion function}.}
    \label{fig: Iteration inversion}
\end{figure}
Finally, the value of the contracted field tensor network is the correct value of the $j$-th variable. We perform this evaluation for each variable. In terms of one equation, the resulting value of the field tensor network to determine the $j$-th variable is $\Omega_j$, which is not a function of the previous variables determined. We have defined that the correct value of that variable is the value of the FTN, so $X_j=\Omega_j$. Making a substitution, we can say
\begin{equation}\label{eq: Master inverse equation}
    \vec{X}=(\Omega_0, \Omega_1, \dots, \Omega_{n-1}).
\end{equation}
This is the \textit{restricted FTNILO inverse function}.

This provides an exact explicit functional (and its equation) to obtain the correct value $\vec{X}$ of every function $f$ satisfying $\vec{Y} = f(\vec{X})$ for every possible value of $\vec{Y}$. However, there are two limitations to this equation. The first is the degenerate case, where we have more than one possible $\vec{X}$ that satisfies the target value. In the next subsection, we will see how to deal with it. The second is the case where there is no value $\vec{X}$ that satisfies $\vec{Y} = f(\vec{X})$. In this case, the first tensor function in \eqref{eq: inversion function} is always zero, making all the following integrals equal to zero. This means that if the result provided by the FTNILO is $\vec{X} = \vec{0}$, we need to check if the result is really that or it means that there is no solution. The fastest way to determine the number of solutions is simply integrating \eqref{eq: final inversion tensor} over all its variables, because
\begin{equation}
    \mathcal{N}=\int_{\mathcal{X}}\mathcal{F}(\vec{x}) \prod_{i=0}^{n-1}dx_i =\int_{\mathcal{X}} \prod_{i=0}^{n-1}\delta(x_i-X_i)\prod_{i=0}^{n-1}dx_i = 1,
\end{equation}
if there is one solution, or otherwise it is
\begin{equation}
    \mathcal{N}=\int_{\mathcal{X}}\mathcal{F}(\vec{x}) \prod_{i=0}^{n-1}dx_i =\int_{\mathcal{X}} 0\prod_{i=0}^{n-1}dx_i  = 0.
\end{equation}
We call this number the \textit{FTNILO number}.

So, this is a previous step to apply before applying the general equation \eqref{eq: Master inverse equation}. With this result, if we can create the logical circuit of the function, and the problem has one solution, the method provides its solution equation. Moreover, if the function $f$ can be computed, it always has a logical circuit associated (for example, the computational circuit). So, every function can have its corresponding explicit equation. It is important to note that the equation is completely explicit because we have constructed it using only the $\vec{Y}$ and $f$ information, even in its computation, so the solution $\vec{X}$ does not appear in the equation.

Finally, we must remember that in tensor networks we can always insert pairs of matrix nodes $A$ and $A^{-1}$ into the edges that connect two tensors, returning the same values after contracting, but being a different tensor network. This means that we can always obtain a new equivalent tensor network from an existing one. Taking into account that there are infinite invertible matrices, there are infinite equivalent tensor networks. Extending this method to the FTN, we have that inserting tensor functions $A$ and $B$ that satisfy
\begin{equation}\label{eq: identity FTN}
    \int_{-\infty}^{\infty} A(x,y)B(y,z)dy = \delta(x-z),
\end{equation}
we can obtain a different and equivalent FTN. This implies that there exists an infinite number of equivalent FTNILO equations. One example of these equations is
\begin{equation}
    A(x,y) = \frac{1}{\sqrt{2\pi}} e^{ixy}, B(y,z) = \frac{1}{\sqrt{2\pi}} e^{-iyz},
\end{equation}
that integrating as in \eqref{eq: identity FTN} results in
\begin{equation}
    \int_{-\infty}^{\infty} \frac{1}{2\pi} e^{ixy} e^{-iyz} dz =
    \frac{1}{2\pi}\int_{-\infty}^{\infty}  e^{i(x-z)y} dz = \delta(x-z),
\end{equation}
that is the equivalent of applying the Fourier Transform and its inverse.

\begin{theorem}[General inversion equation for one solution]\label{th: general inversion}
$ $\\
    Given a vector-valued function $f:\mathcal{X}\rightarrow \mathcal{Y}$, with $\mathcal{X}\subseteq  \mathbb{R}^n$ and $\mathcal{Y}\subseteq \mathbb{R}^m$, and a set of values $\vec{Y}\in  \mathcal{Y}$, if $\exists! \vec{X}\in\mathcal{X} \backslash \ \vec{Y}=f(\vec{X})$, then there always exists an exact explicit function that returns the value $\vec{X}$. The function is
    \begin{equation}
    \vec{X}=(\Omega_0, \Omega_1, \dots, \Omega_{n-1}).
    \end{equation}
    being $\Omega_j$ the FTN associated to the $j$-th variable determination
    \begin{equation}
        \Omega_j = \int_{\mathcal{X}}\int_\mathcal{Y} x_j \Phi(\vec{x},\vec{y})\delta(\vec{y}-\vec{Y})d\vec{y}d\vec{x},
    \end{equation}
    and $\Phi(\vec{x},\vec{y})$ the FTN obtained from tensorizing the logical circuit of the function $f$.
\end{theorem}

\begin{theorem}[Checker of solutions]\label{th: checker}
$ $\\
    Given a vector-valued function $f:\mathcal{X}\rightarrow \mathcal{Y}$, with $\mathcal{X}\subseteq  \mathbb{R}^n$ and $\mathcal{Y}\subseteq \mathbb{R}^m$, and a set of values $\vec{Y}\in  \mathcal{Y}$, the value
    \begin{equation}
        \mathcal{N}=\int_{\mathcal{X}}\int_\mathcal{Y}\Phi(\vec{x},\vec{y})\delta(\vec{y}-\vec{Y}) d\vec{y}d\vec{x}
    \end{equation}
    will be higher than zero only if there exists some $\vec{X}\in\mathcal{X}\ \backslash \ \vec{Y}=f(\vec{X})$, and zero only if there do not exist, being $\Phi(\vec{x},\vec{y})$ the FTN obtained from tensorizing the logical circuit of the function $f$.
\end{theorem}

\begin{theorem}[Infinite number of general inversion equation for one solution]\label{th: infinite inversion}
$ $\\
    Given a vector-valued function $f:\mathcal{X}\rightarrow \mathcal{Y}$, with $\mathcal{X}\subseteq  \mathbb{R}^n$ and $\mathcal{Y}\subseteq \mathbb{R}^m$, and a set of values $\vec{Y}\in \mathcal{Y}$, if $\exists! \vec{X}\in\mathbb{R}^n\ \backslash \ \vec{Y}=f(\vec{X})$, then there always exists an infinite number of exact, explicit functions that return the same value $\vec{X}$.
\end{theorem}

\subsubsection{Optimization Problem}
In the optimization case, to obtain the solution, we have the option of looking directly at the restricted optimization density function obtained from the FTN. However, this would be exactly like solving the problem by brute force, since we would be looking along the entire exponential of the function. To avoid this, we take inspiration from the \textit{Half Partial Trace} method of the MeLoCoToN, and add a One Constant function node to all variables except the one we want to determine. In this way, we will integrate all the variables except the one that we are interested in. Since for a sufficiently high value of $\tau$ there will be a relatively high peak of amplitude only for the optimal value $\vec{X}$ of $\vec{x}\in \mathcal{R}$. This means that, in the limit, the function becomes a Dirac delta centered on the value $\vec{X}$.
\begin{equation}
    \lim_{\tau\rightarrow\infty} \frac{\mathcal{F}(\vec{x},\tau)}{\int_{\mathcal{X}}\mathcal{F}(\vec{x},\tau)d\vec{x}}=
    \lim_{\tau\rightarrow\infty} \frac{e^{-\tau f(\vec{x})}\chi_{\mathcal{R}}(\vec{x})}{\int_{\mathcal{X}}e^{-\tau f(\vec{x})}\chi_{\mathcal{R}}(\vec{x})d\vec{x}}= \delta(\vec{x}-\vec{X}) =  \prod_{i=0}^{n-1}\delta(x_i-X_i).
\end{equation}

Therefore, if we integrate over all variables except $x_j$, we obtain 
\begin{equation}
    \lim_{\tau\rightarrow\infty} F_j(x_j,\tau)=\lim_{\tau\rightarrow\infty} \frac{\int_{\mathcal{X}_{-j}}\mathcal{F}(\vec{x},\tau)\prod_{i=0\backslash i\neq j}^{n-1}dx_i}{\int_{\mathcal{X}}\mathcal{F}(\vec{x},\tau)d\vec{x}} =
    \int_{\mathcal{X}_{-j}} \delta(\vec{x}-\vec{X})\prod_{i=0\backslash i\neq j}^{n-1}dx_i =
    \delta(x_j-X_j).
\end{equation}
We call it the \textit{j-th partial restricted optimization function}.

From this point on, the process is the same as in the inversion case. In this case, each $\Omega_j$ also depends on $\tau$ before the limit, so the equation is the \textit{restricted FTNILO optimization function}
\begin{equation}\label{eq: simple equation optimization with only tau}
    \vec{X}=\lim_{\tau\rightarrow\infty}(\Omega_0(\tau), \Omega_1(\tau), \dots, \Omega_{n-1}(\tau)),
\end{equation}
being the \textit{j-th partial restricted optimization value}
\begin{equation}
    \Omega_j(\tau) = \frac{\int_{\mathcal{X}}\int_\mathcal{X} x_j \Phi(\vec{x},\vec{y},\tau)d\vec{y}d\vec{x}}{\int_{\mathcal{X}}\int_\mathcal{X} \Phi(\vec{x},\vec{y},\tau)d\vec{y}d\vec{x}}.
\end{equation}

Eq.~\ref{eq: simple equation optimization with only tau} is the exact and explicit functional (and its equation) that returns the optimal value of each variable for a function minimization. If we want to maximize, we only need to change the sign of $\tau$. However, the main problem of the formulation is the degenerate case, which we will see in the next subsection.

\begin{theorem}[General optimization equation with no degeneration]\label{th: general optimization}
$ $\\
    Given a real function $f:\mathcal{X}\rightarrow \mathcal{Y}$, with $\mathcal{X}\subseteq  \mathbb{R}^n$ and $\mathcal{Y}\subseteq \mathbb{R}$, with one and only one value $\vec{X}= \arg\min_{\vec{x}} f(\vec{x})$, then there always exists an exact explicit function that returns the value $\vec{X}$. The equation is
    \begin{equation}
    \vec{X}=\lim_{\tau\rightarrow\infty}(\Omega_0(\tau), \Omega_1(\tau), \dots, \Omega_{n-1}(\tau)).
    \end{equation}
    being $\Omega_j$ the FTN associated to the $j$-th variable determination
    \begin{equation}
        \Omega_j(\tau) = \frac{\int_{\mathcal{X}}\int_\mathcal{X} x_j \Phi(\vec{x},\vec{y},\tau)d\vec{y}d\vec{x}}{\int_{\mathcal{X}}\int_\mathcal{X} \Phi(\vec{x},\vec{y},\tau)d\vec{y}d\vec{x}}.
    \end{equation}
\end{theorem}

Obviously, out of the infinite limit, this equation is not exact, but we can approximate it as much as we want by increasing the value of $\tau$. If we do not take the limit (for example, in practical implementations), we must project the previously determined variables into the chosen values. So, these indexes are connected with $\delta(x_i-\hat{X}_i)$ instead of $+(x_i)$, $\hat{X}_i$ being the determined value for the $i$-th variable. This makes the iteration process as shown in Fig.~\ref{fig: Iteration optimization}. We can also use the \textit{Motion Onion} techniques to improve the result without reaching the limit.

\begin{figure}[h]
    \centering
    \includegraphics[width=0.7\linewidth]{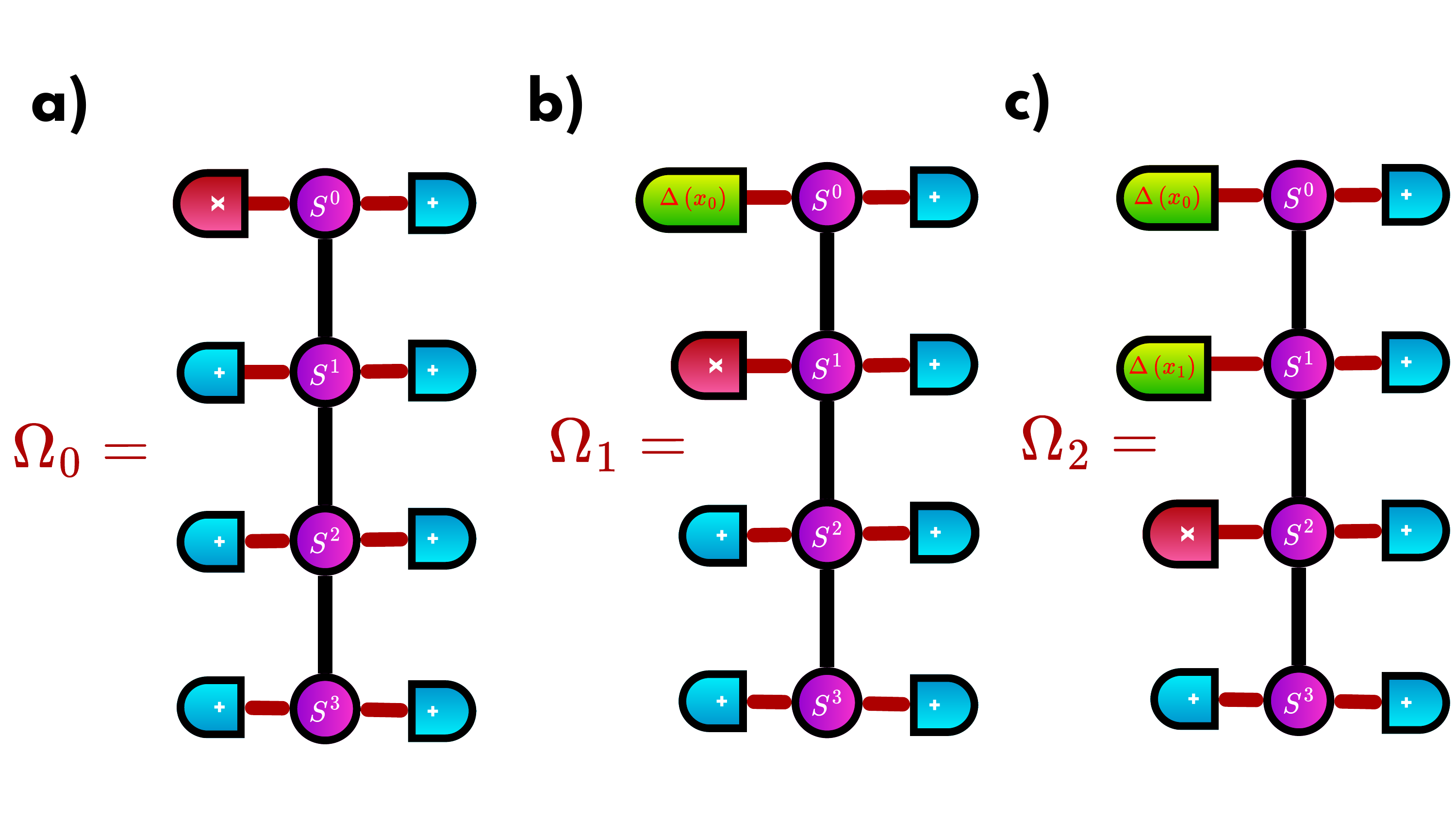}
    \caption{Iteration process to determine first and second variable values in the optimization problem.}
    \label{fig: Iteration optimization}
\end{figure}

In terms of one equation, the resulting value of the field tensor network to determine the $j$-th variable is $\lim_{\tau\rightarrow \infty}\Omega_j(\tau,X_0,X_1,\dots,X_{j-1})$, which is a function of the previous variables determined. We also define that the FTN of the first variable is $\lim_{\tau\rightarrow \infty}\Omega_0(\tau)$, depending on no other variable. We have defined that the correct value of that variable is the value of the FTN, so $X_j=\lim_{\tau\rightarrow \infty}\Omega_j(\tau,X_0,X_1,\dots,X_{j-1})$. Taking a substitution, we can say that

\begin{equation}\label{eq: Master optimization equation}
    X_j=\lim_{\tau\rightarrow \infty}\Omega_j(\tau,\Omega_0(\tau),\Omega_1(\tau,\Omega_0(\tau)),\Omega_2(\tau,\Omega_0(\tau),\Omega_1(\tau,\Omega_0(\tau))),\dots,\Omega_{j-1}(\tau,\dots)).
\end{equation}

\begin{theorem}[General optimization equation with no degeneration out of the limit]\label{th: general optimization limit}
$ $\\
    Given a real function $f:\mathcal{X}\rightarrow \mathcal{Y}$, with $\mathcal{X}\subseteq  \mathbb{R}^n$ and $\mathcal{Y}\subseteq \mathbb{R}$, with one and only one value $\vec{X}= \arg\min_{\vec{x}} f(\vec{x})$, then there always exists an exact explicit function that returns the value $\vec{X}$. The equation is
    \begin{equation}
    X_j=\lim_{\tau\rightarrow \infty}\Omega_j(\tau,\Omega_0(\tau),\Omega_1(\tau,\Omega_0(\tau)),\Omega_2(\tau,\Omega_0(\tau),\Omega_1(\tau,\Omega_0(\tau))),\dots,\Omega_{j-1}(\tau,\dots)).
    \end{equation}
    being $X_j$ the $j$-th element of $\vec{X}$ and $\Omega_j$ the FTN associated with the determination of the $j$-th variable. This function can be arbitrary approximated by the function
    \begin{equation}
        \hat{X}_j=\Omega_j(\tau,\Omega_0(\tau),\Omega_1(\tau,\Omega_0(\tau)),\Omega_2(\tau,\Omega_0(\tau),\Omega_1(\tau,\Omega_0(\tau))),\dots,\Omega_{j-1}(\tau,\dots)),
    \end{equation}
    by increasing the value of the parameter $\tau$.
\end{theorem}

\begin{theorem}[Infinite number of general optimization equations with no degeneration]\label{th: infinite optimization}
$ $\\
    Given a real function $f:\mathcal{X}\rightarrow \mathcal{Y}$, with $\mathcal{X}\subseteq  \mathbb{R}^n$ and $\mathcal{Y}\subseteq \mathbb{R}$, with one and only one value $\vec{X}= \arg\min_{\vec{x}} f(\vec{x})$, there always exists an infinite number of exact explicit equations that return the same value $\vec{X}$.
\end{theorem}

\subsection{Degenerate case}
We have determined the general equations for non-degenerate problems, so we now need to determine them for degenerate problems. In this case, we only need to return to the $\mathcal{F}(\vec{x})$ expressions and take into account that there are multiple points where the function is not zero. This modification forces us to have to sample with the FTNILO number the different regions of the space until we find one with a single solution.

\subsubsection{Inversion Problem}

In the inversion case, the free inversion density function in \eqref{eq: inversion function} and \eqref{eq: inversion function with deltas} is valid, because it takes into account the $\mathcal{N}$ correct values of $\vec{X}$. The following computations are valid up to the expressions \eqref{eq: final inversion tensor without deltas} and \eqref{eq: final inversion tensor}. They are the result of imposing $\vec{Y}=f(\vec{X})$. This means, if we have a set of values $\{\vec{X}_{(k)}, k\in[0,\mathcal{N}-1]\}$ that satisfy $\vec{Y}=f\left(\vec{X}_{(k)}\right)$, the correct expression after the projection is
\begin{equation}\label{eq: final inversion tensor without deltas degenerate}
\mathcal{F}(\vec{x})=
\begin{cases}
    1&\text{ if } \vec{x}\in\left\{\vec{X}_{(k)}, k\in[0,\mathcal{N}-1]\ \backslash \ \vec{Y}=f\left(\vec{X}_{(k)}\right)\right\},\\
    0&\text{ if } \vec{x}\notin\left\{\vec{X}_{(k)}, k\in[0,\mathcal{N}-1]\ \backslash \ \vec{Y}=f\left(\vec{X}_{(k)}\right)\right\}
\end{cases}
\end{equation}
or with the deltas
\begin{equation}\label{eq: final inversion tensor degenerate}
    \mathcal{F}(\vec{x})=\sum_{k=0}^{\mathcal{N}-1}\delta\left(\vec{x}-\vec{X}_{(k)}\right).
\end{equation}
This allows that if $\vec{x}$ has the value of one of the solutions, all deltas except its corresponding one will be zero.

Now, to extract the solution, we use the Half Partial Trace as before
\begin{equation}\label{eq: partial distribution}
    F_j(x_j)=\int_{\mathcal{X}_{-j}} \mathcal{F}(\vec{x}) \prod_{i=0\backslash i\neq j}^{n-1}dx_i =
    \int_{\mathcal{X}_{-j}}  \sum_{k=0}^{\mathcal{N}-1}\delta\left(\vec{x}-\vec{X}_{(k)}\right)\prod_{i=0\backslash i\neq j}^{n-1}dx_i =
    \sum_{k=0}^{\mathcal{N}-1}\delta\left(x_j-\left(\vec{X}_{(k)}\right)_j\right),
\end{equation}
obtaining the sum of the deltas of the solutions. Due to this, if we integrate with the linear functions as before
\begin{equation}
    \Omega_j = \int_{\mathcal{X}_j} x_j F_j(x_j) dx_j =  \int_{\mathcal{X}_j}  x_j \sum_{k=0}^{\mathcal{N}-1}\delta\left(x_j-\left(\vec{X}_{(k)}\right)_j\right) dx_j = \sum_{k=0}^{\mathcal{N}-1}\left(\vec{X}_{(k)}\right)_j,
\end{equation}
resulting in a mixing of the solutions. However, we can use this method to determine all the moments of the distribution \eqref{eq: partial distribution}. The $q$-th raw moment $\mu_{q,j}$ of the distribution $F_j(x_j)$ can be calculated connecting the $q$-th polynomial function $p_q(x)=x^q$ to the free edge of the variable to determine, and it is
\begin{equation}
    \mu_{q,j}= \int_{\mathcal{X}_j} x_j^q F_j(x_j) dx_j = \int_{\mathcal{X}_j} x_j^k \sum_{k=0}^{\mathcal{N}-1}\delta\left(x_j-\left(\vec{X}_{(k)}\right)_j\right) dx_j = \sum_{k=0}^{\mathcal{N}-1}\left(\vec{X}_{(k)}\right)_j^q.
\end{equation}
This allows for the reconstruction of the values $\left(\vec{X}_{(k)}\right)_j$. However, we need a simpler method. We know that the first moment $\mu_{0,j}$ gives us the number of compatible solutions (including degeneration). We will define the $q$-th partial moment of $F_j(x_j)$ up to $r$ as
\begin{equation}
\begin{gathered}
    \mu_{q,j}^-(r)=\int_{-\infty}^{r} x_j^q \sum_{k=0}^{\mathcal{N}-1}\delta\left(x_j-\left(\vec{X}_{(k)}\right)_j\right) dx_j= \int_{\mathcal{X}_j} x_j^q (1-H(x_j-r))\sum_{k=0}^{\mathcal{N}-1}\delta\left(x_j-\left(\vec{X}_{(k)}\right)_j\right) dx_j,\\
    \mu_{q,j}^+(r)=\int_{r}^{\infty} x_j^q \sum_{k=0}^{\mathcal{N}-1}\delta\left(x_j-\left(\vec{X}_{(k)}\right)_j\right) dx_j= \int_{\mathcal{X}_j} x_j^q H(x_j-r)\sum_{k=0}^{\mathcal{N}-1}\delta\left(x_j-\left(\vec{X}_{(k)}\right)_j\right) dx_j,
\end{gathered}
\end{equation}
being $H(x)$ the Heaviside step-function defined as
\begin{equation}
    H(x)=\begin{cases}
        1&\text{if } x\geq 0,\\
        0&\text{if } x< 0.
    \end{cases}
\end{equation}
This is equivalent to replacing the function node $p_q(x_j)$ by $p_q(x_j)(1-H(x_j-r))$ and $p_q(x_j)H(x_j-r)$, respectively. So, we can use the following process for every variable $x_j$:
\begin{protocol}[Inversion Protocol in degenerate cases]\label{protocol: inversion}
$ $\\
    \begin{enumerate}
    \item We calculate the first partial moments $\mu_{0,j}^+(0)$ and $\mu_{0,j}^-(0)$.
    \item 
    \begin{enumerate}
        \item If $\mu_{0,j}^+(0)=1$, we know there is only one solution in this space, so we compute as in the non-degenerate case, but the final integral is only over the positive numbers. We finish.
        \item Otherwise, if $\mu_{0,j}^-(0)=1$, we know there is only one solution in this space, so we compute as in the non-degenerate case, but the final integral is only over the negative numbers. We finish.
        \item Otherwise, if $\mu_{0,j}^+(0)=0$ and $\mu_{0,j}^-(0)=0$, there are no solutions, so we finish.
        \item Otherwise if $\mu_{0,j}^+(0)=\infty$ or $\mu_{0,j}^+(0)=\infty$, we use the renormalized search. We finish.
        \item  If $\mu_{0,j}^+(0)>1$, we calculate $\mu_{1,j}^+(0)$ and $\mu_{2,j}^+(0)$.
        \begin{itemize}
            \item If $\mu_{2,j}^+(0)=\frac{\mu_{1,j}^+(0)^2}{\mu_{0,j}^+(0)}$, we know all the solutions are degenerated in the same value of $x_j$, so the correct value for $x_j$ will be $\frac{\mu_{1,j}^+(0)}{\mu_{0,j}^+(0)}$ because
            \begin{equation}
            \mu_{1,j}^+(0) =\sum_{k=0}^{\mathcal{N}-1}\left(\vec{X}_{(k)}\right)_j=\mathcal{N}\left(\vec{X}_{(0)}\right)_j=\mu_{0,j}^+(0)\left(\vec{X}_{(0)}\right)_j.
            \end{equation}
            We finish.
        \end{itemize}
        \item  If $\mu_{0,j}^-(0)>1$, we calculate $\mu_{1,j}^-(0)$ and $\mu_{2,j}^-(0)$.
            \begin{itemize}
            \item If $\mu_{2,j}^-(0)=\frac{\mu_{1,j}^-(0)^2}{\mu_{0,j}^-(0)}$, we know all the solutions are degenerated in the same value of $x_j$, so the correct value for $x_j$ will be $\frac{\mu_{1,j}^-(0)}{\mu_{0,j}^-(0)}$ because
            \begin{equation}
            \mu_{1,j}^-(0) =-\sum_{k=0}^{\mathcal{N}-1}\left|\left(\vec{X}_{(k)}\right)_j\right|=-\mathcal{N}\left|\left(\vec{X}_{(0)}\right)_j\right|=\mu_{0,j}^-(0)\left(\vec{X}_{(0)}\right)_j.
            \end{equation}
            We finish.
        \end{itemize}
        \item If the the previous steps do not work, we choose a value $r_+>0$ and a value $r_-<0$ and calculate the first partial moments $\mu_{0,j}^+(r_+)$ and $\mu_{0,j}^-(r_-)$.
    \end{enumerate}
    \item
    \begin{enumerate}
        \item If $\mu_{0,j}^+(r_+)=1$, we know there is only one solution in this space, so we compute as in the non-degenerate case, but the final integral is only over the range $[r_+,\infty)$. We finish.
        \item Otherwise, if $\mu_{0,j}^-(r_-)=1$, we know there is only one solution in this space, so we compute as in the non-degenerate case, but the final integral is only over the range $(-\infty,r_-]$. We finish.
        \item Otherwise, if $\mu_{0,j}^+(0)=0$ and $\mu_{0,j}^-(0)=0$, there are no solutions in this range. We choose a lower $r_+>0$ and a higher $r_-<0$, calculate the first partial moments $\mu_{0,j}^+(r_+)$ and $\mu_{0,j}^-(r_-)$,  and return to step 3.
        \item Otherwise, if $|\mu_{0,j}^-(r_-)-\mu_{0,j}^-(r_{-,prev})|=1$, being $r_{-,prev}$ the previous $r_-$, in the region $[r_-,r_{-,prev}]$, there is only one solution. We compute as in the non-degenerate case, but the final integral is only over the range $[r_-,r_{-,prev}]$. We finish.
        \item Otherwise, if $|\mu_{0,j}^+(r_+)-\mu_{0,j}^+(r_{+,prev})|=1$, being $r_{+,prev}$ the previous $r_+$, in the region $[r_+,r_{+,prev}]$ there is only one solution. We compute as in the non-degenerate case, but the final integral is only over the range $[r_+,r_{+,prev}]$. We finish.
        \item Otherwise, if $\mu_{0,j}^+(0)=0$ and $\mu_{0,j}^-(0)=0$, there are no solutions in this range. We choose a lower $r_+>0$ and a higher $r_-<0$, calculate the first partial moments $\mu_{0,j}^+(r_+)$ and $\mu_{0,j}^-(r_-)$,  and return to step 3.
        \item  If $\mu_{0,j}^+(r_+)>1$, we calculate $\mu_{1,j}^+(r_+)$ and $\mu_{2,j}^+(r_+)$.
        \begin{itemize}
            \item If $\mu_{2,j}^+(r_+)=\frac{\mu_{1,j}^+(r_+)^2}{\mu_{0,j}^+(r_+)}$, we know all the solutions are degenerated in the same value of $x_j$, so the correct value for $x_j$ will be $\frac{\mu_{1,j}^+(r_+)}{\mu_{0,j}^+(r_+)}$ because
            \begin{equation}
            \mu_{1,j}^+(r_+) =\sum_{k=0}^{\mathcal{N}-1}\left(\vec{X}_{(k)}\right)_j=\mathcal{N}\left(\vec{X}_{(0)}\right)_j=\mu_{0,j}^+(r_+)\left(\vec{X}_{(0)}\right)_j.
            \end{equation}
            We finish.
        \end{itemize}
        \item  If $\mu_{0,j}^-(r_-)>1$, we calculate $\mu_{1,j}^-(r_-)$ and $\mu_{2,j}^-(r_-)$.
            \begin{itemize}
            \item If $\mu_{2,j}^-(r_-)=\frac{\mu_{1,j}^-(r_-)^2}{\mu_{0,j}^-(r_-)}$, we know all the solutions are degenerated in the same value of $x_j$, so the correct value for $x_j$ will be $\frac{\mu_{1,j}^-(r_-)}{\mu_{0,j}^-(r_-)}$ because
            \begin{equation}
            \mu_{1,j}^-(r_-) =-\sum_{k=0}^{\mathcal{N}-1}\left|\left(\vec{X}_{(k)}\right)_j\right|=-\mathcal{N}\left|\left(\vec{X}_{(0)}\right)_j\right|=\mu_{0,j}^-(r_-)\left(\vec{X}_{(0)}\right)_j.
            \end{equation}
            We finish.
        \end{itemize}
        \item If the two previous steps do not work, we choose a higher value $r_+>0$ if $\mu_{0,j}^+(r_+)>1$ and lower if $\mu_{0,j}^+(r_+)=0$ and a lower value $r_-<0$ if $\mu_{0,j}^-(r_-)>1$ and higher if $\mu_{0,j}^-(r_-)=0$, and calculate the first partial moments $\mu_{0,j}^+(r_+)$ and $\mu_{0,j}^-(r_-)$. We return to step 3.
    \end{enumerate}
\end{enumerate}
\end{protocol}

In this way, we explore sections of the space until we have a region with only one solution. Repetition of the process from the beginning changing the integration regions allows us to obtain all the solutions. The main limitation of this protocol is that, in situations of multiple solutions infinitesimally near, it could fail. For example, the inversion of the ReLU function
\begin{equation}
    ReLU = \max(0,x).
\end{equation}
In this case, if $Y=0$, all negative regions are a valid solution, so the process will never reach a region with only one solution. In this case, the result for an `everything is solution' region $\mathcal{X'}$ will be
\begin{equation}
    \mathcal{N'}=\int_{\mathcal{X'}}\int_\mathcal{Y}\Phi(\vec{x},\vec{y})\delta(\vec{y}-\vec{Y}) d\vec{y}d\vec{x} = \int_{\mathcal{X'}}\int_\mathcal{Y} d\vec{y}\delta(\vec{0})d\vec{x}= \delta(\vec{0})V(\mathcal{X'}),
\end{equation}
being $V(\mathcal{X'})$ the volume of the region. This means that the number obtained will be infinite. This warns us of the situation of infinite solutions. However, if we consider the `everything solution' plus another region, not every point is a solution, but the integral diverges. To determine the correct region, we can do the following. Given two small regions $\mathcal{X'}$ and $\mathcal{X''}\subset \mathcal{X'}$, we calculate
\begin{equation}
    \frac{\mathcal{N'}}{\mathcal{N''}} = \frac{\int_{\mathcal{X'}}\int_\mathcal{Y}\Phi(\vec{x},\vec{y})\delta(\vec{y}-\vec{Y}) d\vec{y}d\vec{x}}{\int_{\mathcal{X''}}\int_\mathcal{Y}\Phi(\vec{x},\vec{y})\delta(\vec{y}-\vec{Y}) d\vec{y}d\vec{x}}
\end{equation}
If both regions are `everything solution' ones or they are not, but they only have solutions in a common `everything solution' region, the result is
\begin{equation}
    \frac{\mathcal{N'}}{\mathcal{N''}} =
    \frac{\int_{\mathcal{X'}}\int_\mathcal{Y} d\vec{y}\delta(\vec{0})d\vec{x}}{\int_{\mathcal{X''}}\int_\mathcal{Y} d\vec{y}\delta(\vec{0})d\vec{x}}=
    \frac{\delta(\vec{0})V(\mathcal{X'}\times)}{\delta(\vec{0})V(\mathcal{X''})}=
    \frac{V(\mathcal{X'})}{V(\mathcal{X''})}.
\end{equation}
If the big one $\mathcal{X'}$ is not an `everything solution' region,
\begin{equation}
    \mathcal{N'}< \delta(\vec{0})V(\mathcal{X'}),
\end{equation}
so, if $\mathcal{X''}$ is an `everything solution' region, then
\begin{equation}
    \frac{\mathcal{N'}}{\mathcal{N''}} =
    \frac{\int_{\mathcal{X'}}\int_\mathcal{Y} d\vec{y}\delta(\vec{0})d\vec{x}}{\int_{\mathcal{X''}}\int_\mathcal{Y} d\vec{y}\delta(\vec{0})d\vec{x}}<
    \frac{\delta(\vec{0})V(\mathcal{X'})}{\delta(\vec{0})V(\mathcal{X''})}=
    \frac{V(\mathcal{X'})}{V(\mathcal{X''})}.
\end{equation}
So we only need to reduce the regions $\mathcal{X''}$ and $\mathcal{X'}$ in the same rhythm until we reach this situation. This will eventually happen due to $\mathcal{X''}\subset \mathcal{X'}$ and it will become an `everything solution' region before $\mathcal{X'}$. We call this process the \textit{renormalized search}.

\begin{theorem}[Number of solutions equation of inversion problems] 
$ $\\
    Given a function $f:\mathcal{X}\rightarrow \mathcal{Y}$, with $\mathcal{X}\subseteq  \mathbb{R}^n$ and $\mathcal{Y}\subseteq \mathbb{R}^m$,  and a set of values $\vec{Y}\in  \mathcal{Y}$, there always exists an exact explicit equation that returns the number $\mathcal{N}$ of values of $\vec{X}_{(k)}\in\mathcal{X}$ that satisfy $\vec{Y}=f(\vec{X}_{(k)})$. This equation is
    \begin{equation}
        \mathcal{N}=\int_{\mathcal{X}}\int_\mathcal{Y}\Phi(\vec{x},\vec{y})\delta(\vec{y}-\vec{Y}) d\vec{y}d\vec{x}=\int_{\mathcal{X}}\mathcal{F}(\vec{x}) d\vec{x}=\mu_{0,j},
    \end{equation}
    being $\Phi(\vec{x},\vec{y})$ the FTN obtained from tensorizing the logical circuit of the function $f$.
\end{theorem}

\begin{theorem}[Protocol of obtaining the solution of inversion problems with degeneration]\label{th: degenerate inversion}
$ $\\
    Given a function $f:\mathcal{X}\rightarrow \mathcal{Y}$, with $\mathcal{X}\subseteq  \mathbb{R}^n$ and $\mathcal{Y}\subseteq \mathbb{R}^m$, and a set of values $\vec{Y}\in\mathcal{Y}$, there always exists an exact, explicit function $\mathcal{F}(\vec{x})$ given by the logical FTN from the tensorization of the logical circuit of the problem that, with the protocol \ref{protocol: inversion}, returns the values of one $\vec{X}_{(k)}\in\mathbb{R}^n$ that satisfies $\vec{Y}=f(\vec{X}_{(k)})$.
\end{theorem}

\begin{theorem}[Protocol of obtaining all the solutions of inversion problems with degeneration]\label{th: degenerate inversion}
$ $\\
    Given a function $f:\mathcal{X}\rightarrow \mathcal{Y}$, with $\mathcal{X}\subseteq  \mathbb{R}^n$ and $\mathcal{Y}\subseteq \mathbb{R}^m$, and a set of values $\vec{Y}\in\mathcal{Y}$, there always exists an exact, explicit function $\mathcal{F}(\vec{x})$ given by the logical FTN from the tensorization of the logical circuit of the problem that, with the protocol \ref{protocol: inversion} and repeated in different regions, returns all the values $\vec{X}_{(k)}\in\mathcal{X}$ that satisfy $\vec{Y}=f(\vec{X}_{(k)})$.
\end{theorem}

Another way to solve this problem is by simply checking the non-zero position of a one-dimensional function. If we know that we can obtain the partial distribution function of \eqref{eq: partial distribution}, and it is of the form $\sum_{k=0}^{\mathcal{N}-1}\delta\left(x_j-\left(\vec{X}_{(k)}\right)_j\right)$, it is a one-dimensional function. We can simply make a search using the one-dimensional argmax, as shown in Fig.~\ref{fig: Iteration inverse multiple},
\begin{equation}
X_j = \arg\max_{x_j} F_j(x_j).
\end{equation}

\begin{figure}[h]
    \centering
    \includegraphics[width=0.5\linewidth]{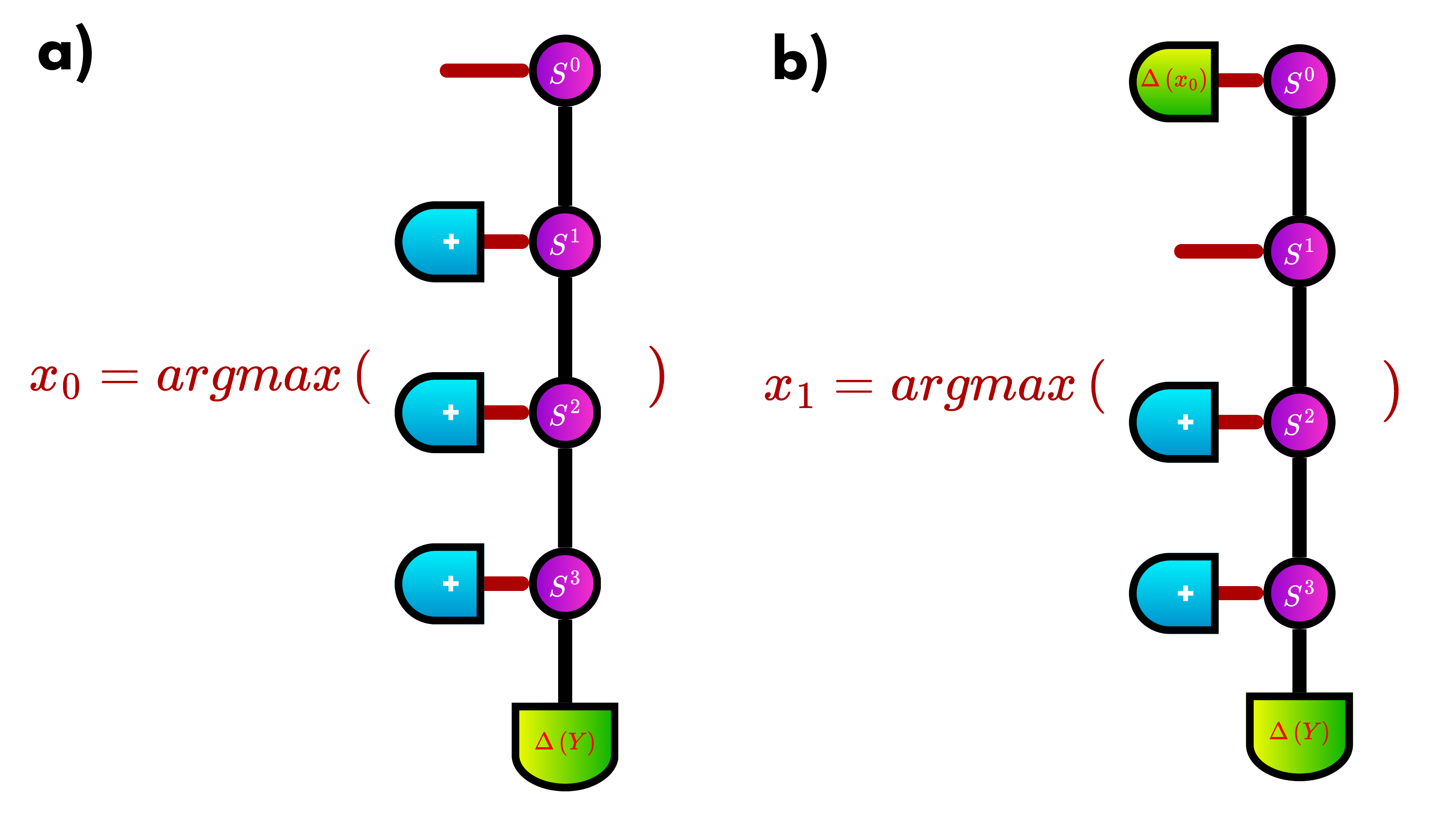}
    \caption{Example of iteration process to determine first and second variable values in an inversion problem with multiple solutions.}
    \label{fig: Iteration inverse multiple}
\end{figure}

\subsubsection{Optimization Problem}

In the optimization problem, we have a similar situation. We know that, if we have a set of values $\{\vec{X}_{(k)}, k\in[0,\mathcal{N}-1]\}$ corresponding with the minimum value of the function, we will have
\begin{equation}
    \lim_{\tau\rightarrow\infty} \frac{\mathcal{F}(\vec{x},\tau)}{\int_{\mathcal{X}}\mathcal{F}(\vec{x},\tau)d\vec{x}}=
    \lim_{\tau\rightarrow\infty} \frac{e^{-\tau f(\vec{x})}\chi_{\mathcal{R}}(\vec{x})}{\int_{\mathcal{X}}e^{-\tau f(\vec{x})}\chi_{\mathcal{R}}(\vec{x})d\vec{x}}= \sum_{k=0}^{\mathcal{N}-1}\kappa_k\delta(\vec{x}-\vec{X_{(k)}}),
\end{equation}
being $\kappa_k$ the proportion of the $k$-th delta in the global distribution, due to the second derivative of the function $f$ in the minimum points. We cannot suppose to know these values because they depend on the solution we are looking for.

From this point, the reasonings are the same as in the inversion case, taking into account the proportions. The result is
\begin{equation}
        X_j(\tau) = \arg\max_{x_j}\left(\lim_{\tau\rightarrow\infty}\frac{\int_{\mathcal{X}_{-j}}\int_\mathcal{X}\Phi(\vec{x},\vec{y},\tau)d\vec{y}\prod_{i=0\backslash i\neq j}^{n-1}dx_i}{\int_{\mathcal{X}}\int_\mathcal{X} \Phi(\vec{x},\vec{y},\tau)d\vec{y}d\vec{x}}\right).
    \end{equation}

We can only use the argmax function and not the moments because the unknown value of the proportions $\kappa_k$ modify the information we can extract from them. This can be expressed as Fig.~\ref{fig: Iteration optimization multiple}.

\begin{figure}[h]
    \centering
    \includegraphics[width=0.7\linewidth]{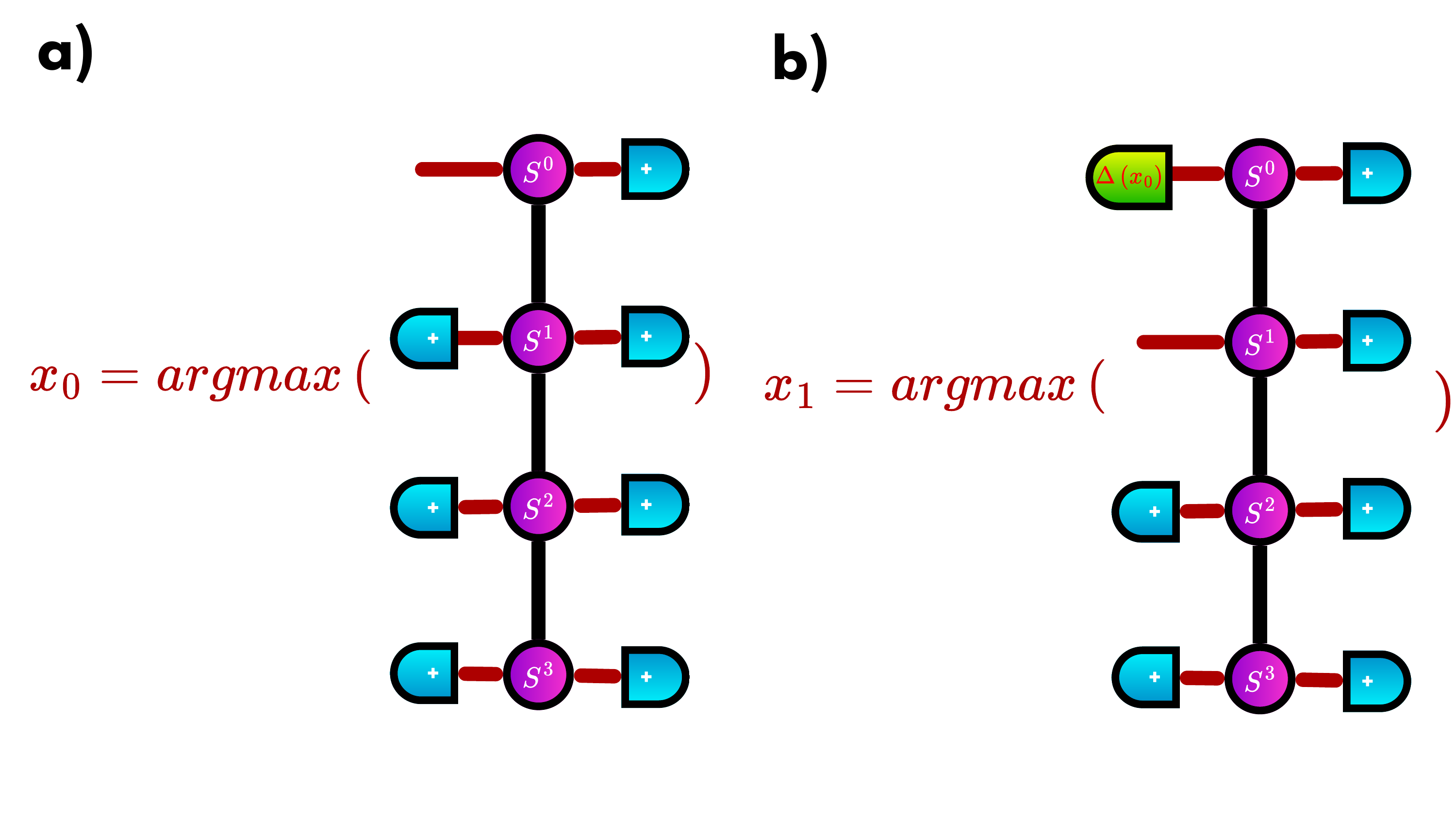}
    \caption{Iteration process to determine first and second variable values in the optimization problem with multiple solutions.}
    \label{fig: Iteration optimization multiple}
\end{figure}

\subsection{Approximations of the equation}
As we presented in \eqref{eq: identity FTN}, the delta functions can be decomposed into integrals of a product of functions. Both the inversion and the optimization functions are composed of products of only delta functions, with other ones in the optimization case. If we decompose the delta functions into other functions, we can rearrange the equation, allowing new possibilities to study their properties.

Moreover, we know that the delta function can be obtained from several limits. Some examples are
\begin{equation}
\delta(x) = \lim_{\sigma \to 0} \frac{1}{\sqrt{2\pi}\sigma} e^{-\frac{x^2}{2\sigma^2}}
\end{equation}
\begin{equation}
    \delta(x) = \lim_{\gamma \to 0} \frac{1}{\pi} \frac{\gamma}{x^2 + \gamma^2}
\end{equation}
\begin{equation}
    \delta(x) = \lim_{L \to \infty} \frac{\sin(Lx)}{\pi x}
\end{equation}
\begin{equation}
    \delta(x) = \lim_{\epsilon \to 0} \frac{1}{2\epsilon} \chi_{[-\epsilon, \epsilon]}(x)
\end{equation}
\begin{equation}
    \chi_{[-\epsilon, \epsilon]}(x) = 
\begin{cases}
1, & |x| \leq \epsilon \\
0, & \text{otherwise}
\end{cases}
\end{equation}

\begin{equation}
    \delta(x) = \lim_{\lambda \to \infty} \lambda e^{-\lambda x} H(x)
\end{equation}

This means that we decompose the deltas in the limit expressions, operate the decomposed functions, and take the limit. This allows us to describe the FTNILO functions in an approximated way.

\newpage
\section{Tensor Network limit and MeLoCoToN recover}
We have developed the formalism for continuous variables and signals taking inspiration from the MeLoCoToN formalism. The similarities between both formalisms give us an idea of how to connect them and the recovery of the MeLoCoToN as a particular case of FTNILO.

To understand the process, we take a three-operators logical FTN
\begin{equation}
    \Phi = \int A(x)B(y)C(z)\delta(x'-f(x))\delta(y'-g(x',y))\delta(z-h(y')) dxdydzdx'dy'dz.
\end{equation}
This FTN receives the input $x,y$ and returns a conditioned output $z$, making use of the signals $x'$ and $y'$. In this case, the inputs, the outputs, and the signals are continuous. If we want to impose the fact that they can only take discrete values, we could restrict the integration region. Moreover, we can also use indicator functions to impose the discrete values
\begin{equation}
    \Phi = \int \chi_\mathbb{N}(x,y,z,x',y') A(x)B(y)C(z)\delta(x'-f(x))\delta(y'-g(x',y))\delta(z-h(y')) dxdydzdx'dy'dz,
\end{equation}
being in this case $f,g,h$ discrete functions from natural numbers to natural numbers.

If we change the integration region to the natural numbers, we can change the integrals for summations and the Dirac deltas for Kronecker deltas, having
\begin{equation}
    \Phi = \sum_{x,y,z,x',y'} A(x)B(y)C(z)\delta_{x',f(x)}\delta_{y',g(x',y)}\delta_{z,h(y')}.
\end{equation}
Being $A,B,C$ functions evaluated only in natural numbers, they can be expressed as tensors
\begin{equation}
    \Phi = \sum_{x,y,z,x',y'} A_xB_yC_z\delta_{x',f(x)}\delta_{y',g(x',y)}\delta_{z,h(y')}.
\end{equation}
If we define $A'_{x,x'}=A_x\delta_{x',f(x)}$, $B'_{y,x',y'}=B_y\delta_{y',g(x',y)}$ and $C'_{z,y'}=C_z\delta_{z,h(y')}$, we get
\begin{equation}
    \Phi = \sum_{x,y,z,x',y'} A'_{x,x'}B'_{y,x',y'}C'_{z,y'},
\end{equation}
which is easily identifiable as a tensor network. This is precisely a logical tensor network, the main ingredient of the MeLoCoToN.

This means that MeLoCoToN can be recovered from FTNILO by integration only over natural numbers and by having only integer-valued functions. So, all the consequences of MeLoCoToN are valid for FTNILO.

From this point on, we can hybridize these two formalisms. The motivation is to give an efficient representation for circuits or functions with discrete internal signals. For example, functions with `if' statements. Taking as example the previous FTN, we could have the situation of $x'$ being a binary signal indicating, for example, if the integer part of $x$ is odd or even. In this case, we only need to change the integrals over the discrete variables to summations and their Dirac deltas to Kronecker deltas. In this case, the equation is
\begin{equation}
    \Phi = \sum_{x'}\int A(x)B(y)C(z)\delta_{x',f(x)}\delta(y'-g(x',y))\delta(z-h(y')) dxdydzdy'dz.
\end{equation}

This approach could be more convenient for different types of problems than a pure FTNILO or MeLoCoToN approach.

\newpage
\section{Mathematical implications}
We can see that this formalism has several mathematical implications that we will expose here.

\subsection{Function universal inversion}
The first implication, due to Theorem \ref{th: degenerate inversion}, is that if a function $f$ has only one possible $\vec{X}$ that returns the desired $\vec{Y}$, there exists an exact explicit equation to obtain $\vec{X}$. Extending it for injective functions, we can get the following theorems.
\begin{theorem}[Injective function inverse]
$ $\\
    Every injective function $f:\mathcal{X}\rightarrow \mathcal{Y}$, with $\mathcal{X}\subseteq  \mathbb{R}^n$ and $\mathcal{Y}\subseteq \mathbb{R}^m$, is invertible in its image $\mathcal{Y}$ and its inverse function is the \textit{FTNILO inversion function}
    \begin{equation}
    \vec{x}=f^{-1}(\vec{y})=(\Omega_0(\vec{y}), \Omega_1(\vec{y}), \dots, \Omega_{n-1}(\vec{y})).
    \end{equation}
    being $\Omega_j(\vec{y})$ the \textit{j-th partial FTNILO inversion function}
    \begin{equation}
        \Omega_j(\vec{y})=\int_{\mathcal{X}}x_j\Phi(\vec{x}, \vec{y}) d\vec{x}.
    \end{equation}
\end{theorem}

\begin{proof}
    Taking into account that an injective function has only one possible value $\vec{x}$ for every $\vec{y}$, and not every $\vec{y}$ necessarily has a corresponding $\vec{X}$, we can apply Theorem \ref{th: degenerate inversion} to every point of the image. This is equivalent to having an FTN of \eqref{eq: Master inverse equation} without the $\Delta(\vec{Y})$ function imposing one specific result, as shown in \ref{fig: inversion injective}. Starting from \eqref{eq: inversion function with deltas}, if we do not integrate with the $\vec{y}$, we can follow the other same steps,
    \begin{equation}
         \varphi(x_0)=\int_{\mathcal{X}_{-0}}\Phi(\vec{x}, \vec{y}) \prod_{i=1}^{n-1}dx_i =
         \int_{\mathcal{X}_{-0}} \delta(\vec{y}-f(\vec{x}))\prod_{i=1}^{n-1}dx_i =
        \delta(\vec{y}-f(x_0,(f^{-1}(\vec{y}))_{-0})),
    \end{equation}
    being $\vec{a}_{-i}$ all the components of $\vec{a}$ but the $i$-th one.
    Integrating over $x_0$ we get
    \begin{equation}
        \Omega_0(\vec{y})=\int_{\mathcal{X}_{0}}x_0\varphi(x_0) dx_0 = \int_{\mathcal{X}_{0}}x_0\delta(\vec{y}-f(x_0,(f^{-1}(\vec{y}))_{-0})) dx_0 = (f^{-1}(\vec{y}))_{0}.
    \end{equation}
    Analogous for the other variables.
    \begin{equation}
         \varphi(x_j)=\int_{\mathcal{X}_{-j}}\Phi(\vec{x}, \vec{y}) \prod_{i=0\backslash i\neq j}^{n-1}dx_i =
         \int_{\mathcal{X}_{-j}} \delta(\vec{y}-f(\vec{x}))\prod_{i=0\backslash i\neq j}^{n-1}dx_i =
        \delta(\vec{y}-f(x_j,(f^{-1}(\vec{y}))_{-j})),
    \end{equation}
    \begin{equation}\label{eq: omega y def}
        \Omega_j(\vec{y})=\int_{\mathcal{X}_{j}}x_j\varphi(x_j) dx_j = \int_{\mathcal{X}_{j}}x_j\delta(\vec{y}-f(x_j,(f^{-1}(\vec{y}))_{-j})) dx_j = (f^{-1}(\vec{y}))_{j}.
    \end{equation}
\end{proof}

\begin{figure}[h]
    \centering
    \includegraphics[width=0.7\linewidth]{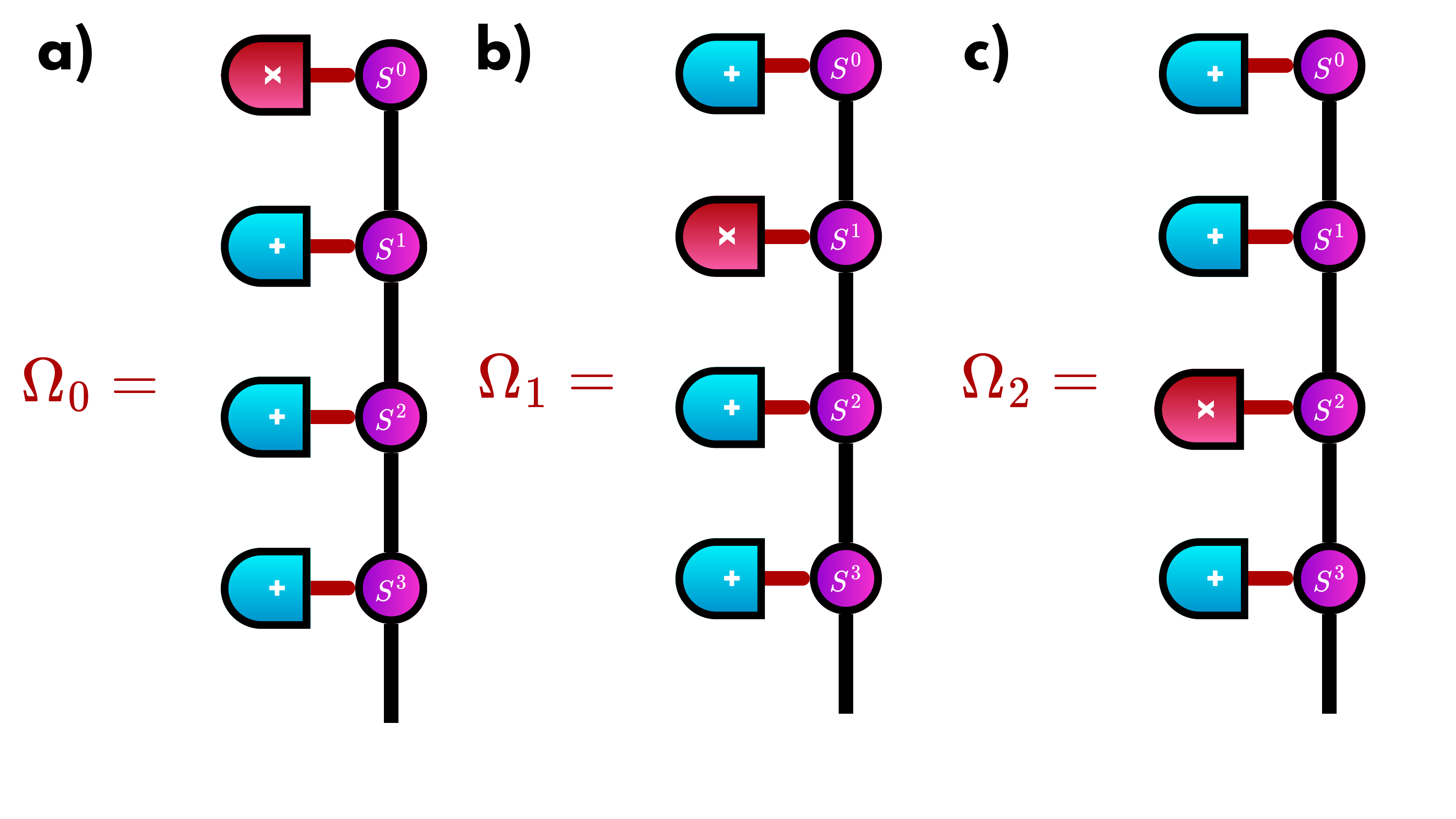}
    \caption{Iteration process to determine first and second variable values in the inversion of the function \eqref{eq: first inversion function}.}
    \label{fig: inversion injective}
\end{figure}

\begin{figure}[h]
    \centering
    \includegraphics[width=0.7\linewidth]{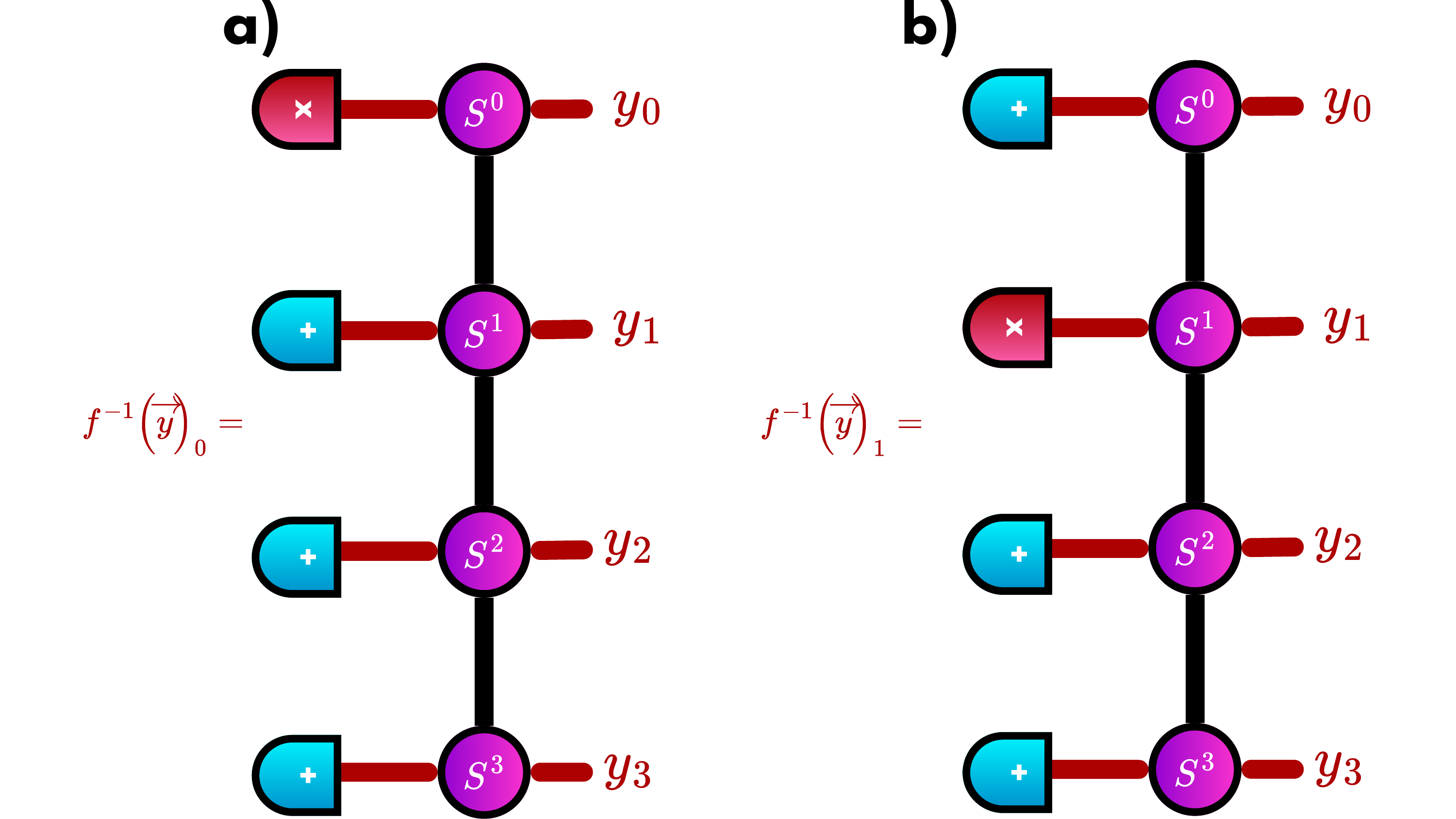}
    \caption{Example of iteration process to determine first and second variable values in the inversion of a vector-valued function.}
    \label{fig: inversion injective 2}
\end{figure}

\begin{lemma}[Bijective function inverse]
$ $\\
    Every bijective function $f:\mathcal{X}\rightarrow \mathcal{Y}$, with $\mathcal{X}\subseteq  \mathbb{R}^n$ and $\mathcal{Y}\subseteq \mathbb{R}^m$ is invertible, and its inverse equation is
    \begin{equation}
    \vec{x}=f^{-1}(\vec{y})=(\Omega_0(\vec{y}), \Omega_1(\vec{y}), \dots, \Omega_{n-1}(\vec{y})).
    \end{equation}
    being $\Omega_i(\vec{y})$ the FTN associated with the determination of the $i$-th variable.
\end{lemma}

\begin{lemma}[Non-injective function inversion]
$ $\\
     Every non-injective function $f:\mathcal{X}\rightarrow \mathcal{Y}$, with $\mathcal{X}\subseteq  \mathbb{R}^n$ and $\mathcal{Y}\subseteq \mathbb{R}^m$, has a protocol described in \ref{protocol: inversion} that makes use of \eqref{eq: omega y def} to obtain its inverse values.
\end{lemma}

\begin{lemma}[$\sigma$-algebra of an injective function]
    Given an injective function $f:\mathcal{X}\rightarrow \mathcal{Y}$, and $B$ being a $\sigma$-algebra of $\mathcal{Y}$, its $\sigma$-algebra is
    \begin{equation}
        \sigma(f) = \{ f^{-1}(S): S\in B \}=\{ \vec{\Omega}(S): S\in B \}.
    \end{equation}
\end{lemma}

We have seen that every injective function can be inverted with the FTNILO formalism. With the same ideas, we can create a new representation of the initial function, allowing its analysis from a different point of view.

\begin{theorem}[FTNILO transformation]
    Given an injective function $f:\mathcal{X}\rightarrow \mathcal{Y}$, with $\mathcal{X}\subseteq  \mathbb{R}^n$ and $\mathcal{Y}\subseteq \mathbb{R}^m$, we can transform it by the equation
    \begin{equation}
        f(\vec{x})_j = \int_\mathcal{Y} y_j\Phi(\vec{x},\vec{y})d\vec{y}.
    \end{equation}
\end{theorem}
\begin{proof}
    If the function is injective
    \begin{equation}
        \Phi(\vec{x},\vec{y}) = \delta(\vec{y}-f(\vec{x})),
    \end{equation}
    then
    \begin{equation}
        \int_\mathcal{Y} y_j\Phi(\vec{x},\vec{y})d\vec{y} = \int_\mathcal{Y} y_j\delta(\vec{y}-f(\vec{x}))d\vec{y}=\int_\mathcal{Y_j} y_j\delta(y_j-f(\vec{x})_j)dy_j = f(\vec{x})_j.
    \end{equation}
\end{proof}

\begin{figure}[h]
    \centering
    \includegraphics[width=0.7\linewidth]{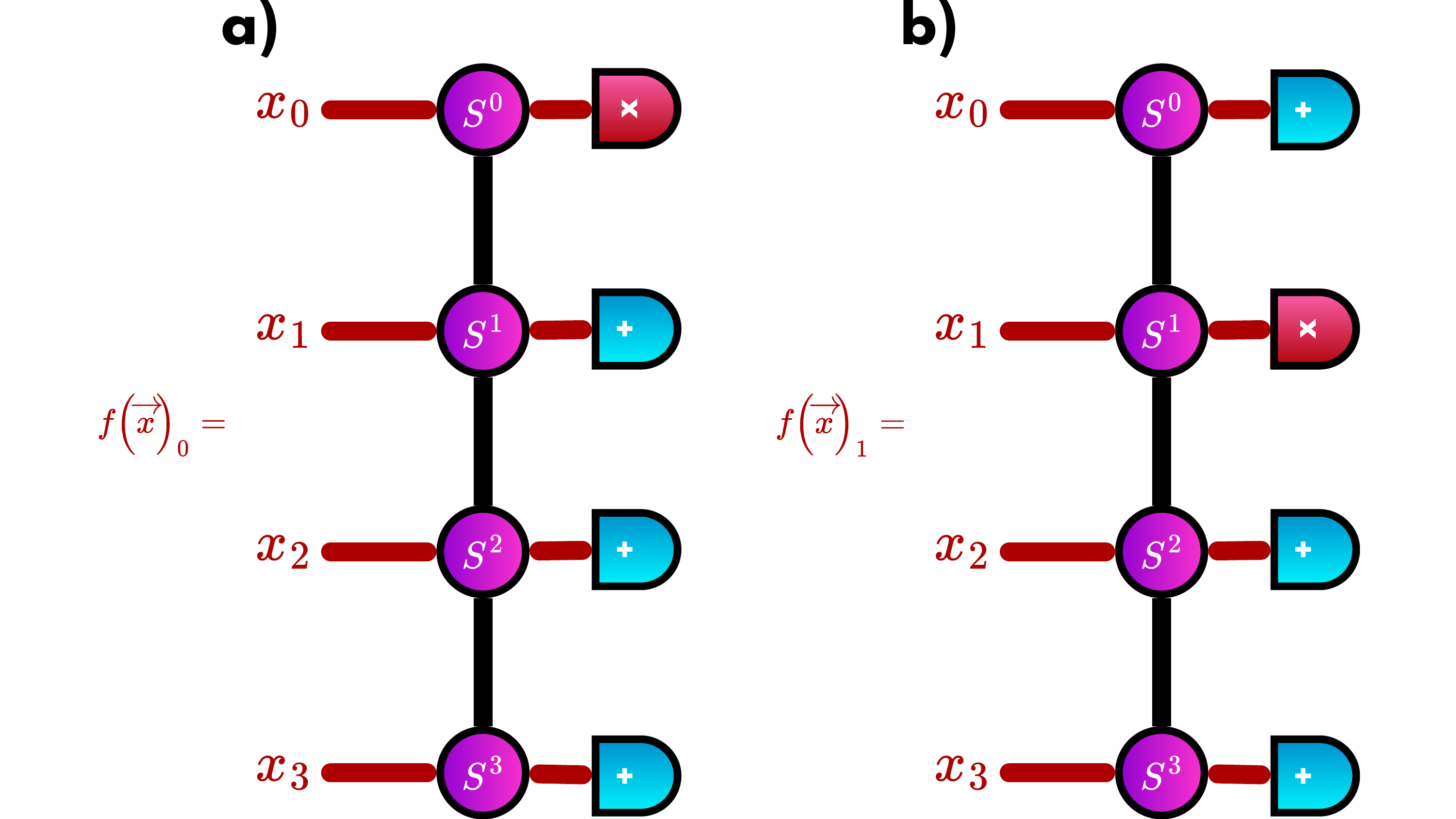}
    \caption{Example of iteration process to determine first and second output values of a vector-valued injective function through the FTNILO transformation.}
    \label{fig: inversion forward}
\end{figure}

\subsection{Cryptographic weakness of every function}

All of these theorems and lemmas have several implications in cryptography.
\begin{lemma}[Cryptographic function break (injective)]
    $ $\\
    Every cryptographic protocol given by an injective encryption function $f:\mathcal{X}\rightarrow \mathcal{Y}$, with $\mathcal{X}\subseteq  \mathbb{R}^n$ and $\mathcal{Y}\subseteq \mathbb{R}^m$, with a secret value $\vec{X}\in\mathcal{X}$ and an encrypted public value $\vec{Y}\in\mathcal{Y}\ \backslash\ \vec{Y}=f(\vec{X})$ has an exact equation that returns the secret value from the public one by \eqref{eq: Master inverse equation}. In other words, every cryptographic protocol based on injective functions can be broken.
\end{lemma}

\begin{lemma}[Cryptographic function break (non-injective)]
    $ $\\
    Every cryptographic protocol given by a non-injective encryption function $f:\mathcal{X}\rightarrow \mathcal{Y}$, with $\mathcal{X}\subseteq  \mathbb{R}^n$ and $\mathcal{Y}\subseteq \mathbb{R}^m$, with a secret value $\vec{X}\in\mathcal{X}$ and an encrypted public value $\vec{Y}\in\mathcal{Y}\ \backslash\ \vec{Y}=f(\vec{X})$ has an exact protocol that returns the secret value from the public one by \ref{protocol: inversion}. In other words, every cryptographic protocol based on non-injective functions can be broken.
\end{lemma}
\begin{corollary}[Cryptographic function break]
    $ $\\
    Every cryptographic protocol given by an encryption function $f:\mathcal{X}\rightarrow \mathcal{Y}$, with $\mathcal{X}\subseteq  \mathbb{R}^n$ and $\mathcal{Y}\subseteq \mathbb{R}^m$, with a secret value $\vec{X}\in\mathcal{X}$ and an encrypted public value $\vec{Y}\in\mathcal{Y}\ \backslash\ \vec{Y}=f(\vec{X})$ has an exact algorithm that returns the secret value from the public one. In other words, every cryptographic protocol based on functions can be broken.
\end{corollary}

\newpage
\section{Riemann hypothesis}

To evaluate the potential of this formalism, we choose to approach the Riemann hypothesis. The Riemann hypothesis states that all the non-trivial zeros of the Riemann Zeta Function are inside the critical line $Re(s)=\frac{1}{2}$. In other words, we need to count the number of zeros for this function outside the critical line. One way is to make use of the previously presented FTNILO number.

The first step of its resolution is the construction of the logical circuit that computes the Riemann Zeta function. It entirely depends on the region of $s$ we want to explore. The regions we will explore, the interesting ones, follow the same structure of a sum of independent terms, related only by $s$. So, we can create a circuit that receives each part of $s$ and transmits it to every operator that implements one step of the summation. Each operator transmits the summation up to its term to the following one, and the last one outputs the value of the function. There are several schemes that we can follow, each with its own advantages, and we will explore them with the easiest region.

After the construction of the circuit, we tensorize it into a field tensor network, and post-select the output of the function to $0$. To evaluate a particular region, we need to restrict the input values of $s$. Finally, we take the infinite limit. The resulting equation will give us the number of zeros in the region of interest.

\subsection{Riemann Function for $Re(s)>1$}
The Riemann Zeta Function for $Re(s)>1$ is 
\begin{equation}
    \zeta (s)=\sum_{n=1}^{\infty} \frac{1}{n^s},\ s\in \mathbb{C}.
\end{equation}

This expression is an infinite sum of independent terms that depend all on the same complex value $s$. Our formalism only works with real numbers, so in the circuit we will take in two separated signals the real and imaginary parts of all the values.

The first thing we need is a set of operators $S$ that receive the real and imaginary parts of $s$ and compute $\frac{1}{n^s}$. This operator also receives the previous value of the partial summation $z_{n-1}$, also in real and imaginary parts, and outputs the real and imaginary parts of $z_{n-1}+\frac{1}{n^s}$. This operator tensorized can be done with Dirac delta functions
\begin{equation}
    S_n(z_n,z_{n-1},x'_n,y'_n) = \delta\left(z_n - \left(z_{n-1}+Re\left(\frac{1}{n^{x_n'+iy_n'}}\right)\right)\right)
    \delta\left(z'_n - \left(z'_{n-1}+Im\left(\frac{1}{n^{x_n'+iy_n'}}\right)\right)\right),
\end{equation}
being $x'$ and $y'$ the real and imaginary parts of $s$, and $z$ and $z'$ the real and imaginary parts of the partial summation of the Riemann Zeta Function.

The second thing we need is the transmission of the values of $s$. For this, we only need to communicate the signal of each part through a chain of operators that transmit the same output as their input. The tensorization of those operators is
\begin{equation}
    \begin{gathered}
        \delta(x_{n}-x_{n-1})\delta(x'_{n}-x_{n-1}),\\
        \delta(y_{n}-y_{n-1})\delta(y'_{n}-y_{n-1}).
    \end{gathered}
\end{equation}

\begin{figure*}
    \centering
    \includegraphics[width=\linewidth]{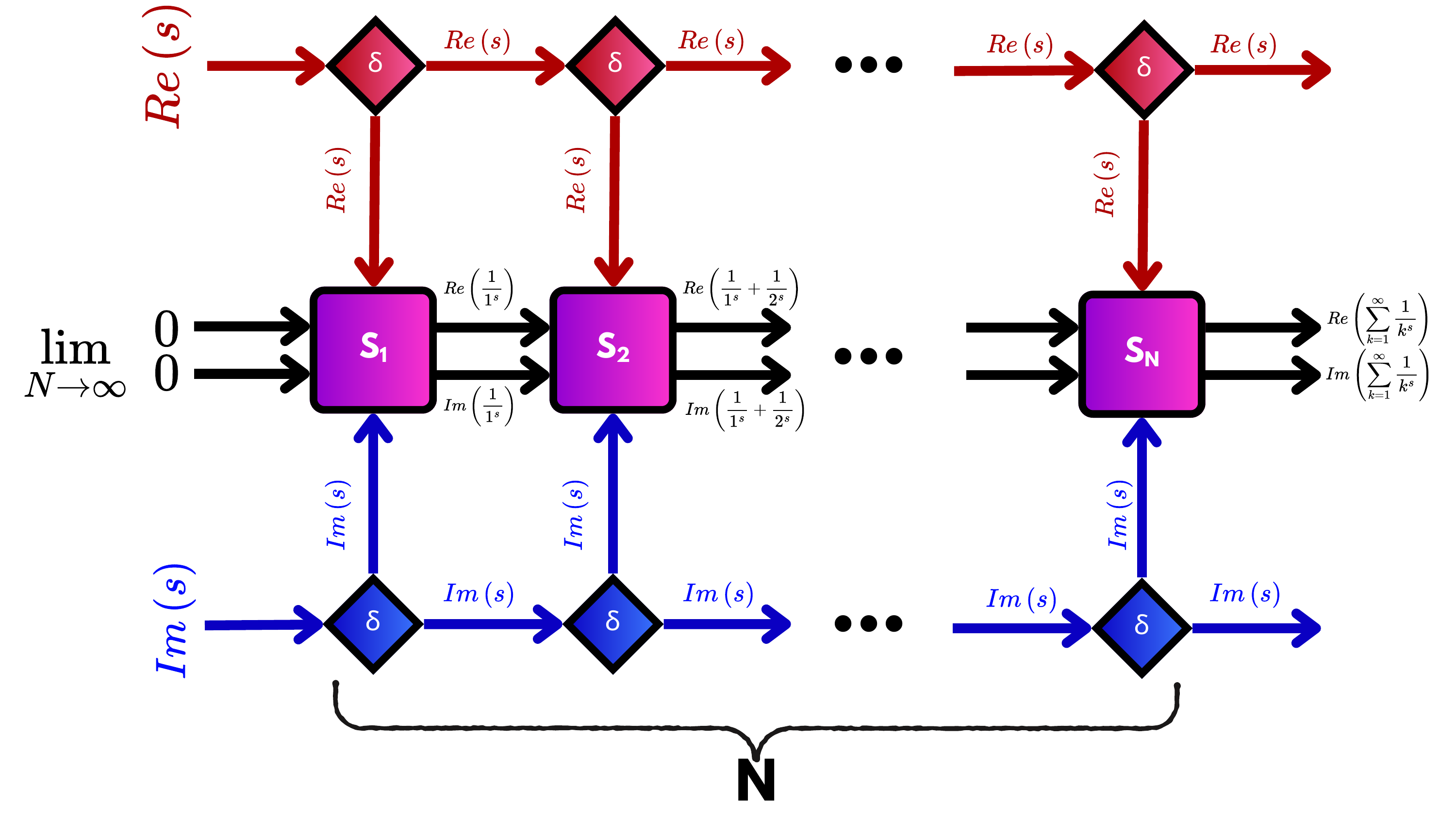}
    \caption{Logical Circuit for the Riemann Zeta Function.}
    \label{fig: Riemann circuit}
\end{figure*}

\begin{figure*}
    \centering
    \includegraphics[width=\linewidth]{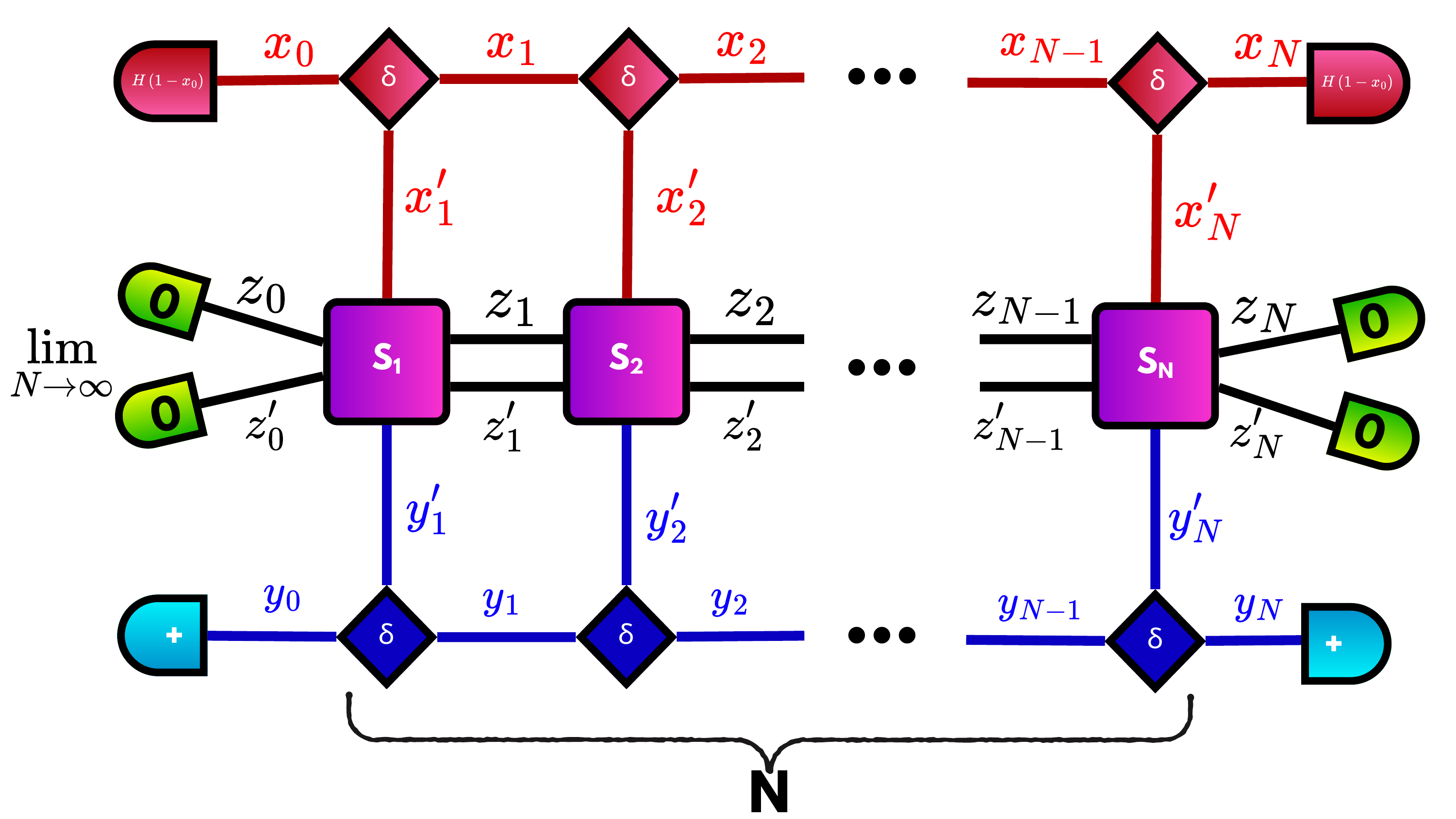}
    \caption{FTNILO of the Riemann Zeta Function zeros. This FTN returns the amount of zeros of the Riemann Zeta Function in the $Re(s)>1$ region.}
    \label{fig: Riemann TN}
\end{figure*}

This is the \textit{Riemann lineal circuit}, shown in Fig.~\ref{fig: Riemann circuit}. To complete the construction, after the tensorization of the circuit, we need to impose the value $0$, so we connect to the first $S$ function and to the last one the following functions
\begin{equation}
        \delta(z_0),\ \delta(z'_0),\ \delta(z_{n}),\ \delta(z'_{n}).
\end{equation}

Finally, we need to impose the region of $s$. This expression is only valid for $Re(s)>1$, where there are no trivial zeros or the critical line, so we can connect a One Constant function to the imaginary part of $s$, and a $H(1-x_0)$ and $H(1-x_{N})$ to the real part. The obtained FTN is shown in Fig.~\ref{fig: Riemann TN}. This is the \textit{Riemann linear field tensor network}.

The equation of this field tensor network is
\begin{equation}\label{eq: Riemann 1}
    \begin{gathered}
        \mathcal{N} =\lim_{N\rightarrow\infty} \int_{\mathbb{R}^{2N+1}}\int_{\mathbb{R}^{2N+1}}\int_{\mathbb{R}^{2(N+1)}}
        H(1-x_{0})H(1-x_{N})
        \delta(z_0)\delta(z'_0)\times\\
        \times\prod_{n=1}^{N}\left(\delta(x_{n}-x_{n-1})\delta(x'_{n}-x_{n-1})\delta(y_{n}-y_{n-1})\delta(y'_{n}-y_{n-1})S_n(z_n,z_{n-1},x'_n,y'_n)\right)\times\\
        \times\delta(z_{N}) \delta(z'_{N})
        d\vec{x}d\vec{x'}d\vec{y}d\vec{y'}d\vec{z}d\vec{z'},
    \end{gathered}
\end{equation}
that developed is
\begin{equation}\label{eq: Riemann 1 developed}
    \begin{gathered}
        \mathcal{N} =\lim_{N\rightarrow\infty} \int_{\mathbb{R}^{2N+1}}\int_{\mathbb{R}^{2N+1}}\int_{\mathbb{R}^{2(N+1)}}
        H(1-x_{0})H(1-x_{N})
        \delta(z_0)\delta(z'_0)\times\\
        \times\prod_{n=1}^{N}\left(
        \delta(x_{n}-x_{n-1})\delta(x'_{n}-x_{n-1})\delta(y_{n}-y_{n-1})\delta(y'_{n}-y_{n-1})\times\right.\\
        \left.\times\delta\left(z_n - \left(z_{n-1}+Re\left(\frac{1}{n^{x_n'+iy_n'}}
        \right)\right)\right)
    \delta\left(z'_n - \left(z'_{n-1}+Im\left(\frac{1}{n^{x_n'+iy_n'}}\right)\right)\right)\right)\times\\
        \times\delta(z_{N}) \delta(z'_{N})
        d\vec{x}d\vec{x'}d\vec{y}d\vec{y'}d\vec{z}d\vec{z'}.
    \end{gathered}
\end{equation}
To improve the symmetry of this FTN, we can take into account that the input of the circuit must be the same as the output, so we can connect the input and the output, as shown in Fig.~\ref{fig: Riemann TN donut}. This is the \textit{Riemann Donut Field Tensor Network}. The only change we need to implement is that the $H(1-x_0)$ must be in all the chain of tensors. This can be implemented in the $S$ tensors, which makes them not allow $x<0$, changing every delta tensor function in the real part by $\delta(x_{n}-x_{n-1})\delta(x'_{n}-x_{n-1})H(1-x_{n-1})$, or changing the integration region. All of them are equivalent, and they can also be applied to the previous equation. In this case, we will take the former one.

\begin{figure*}
    \centering
    \includegraphics[width=\linewidth]{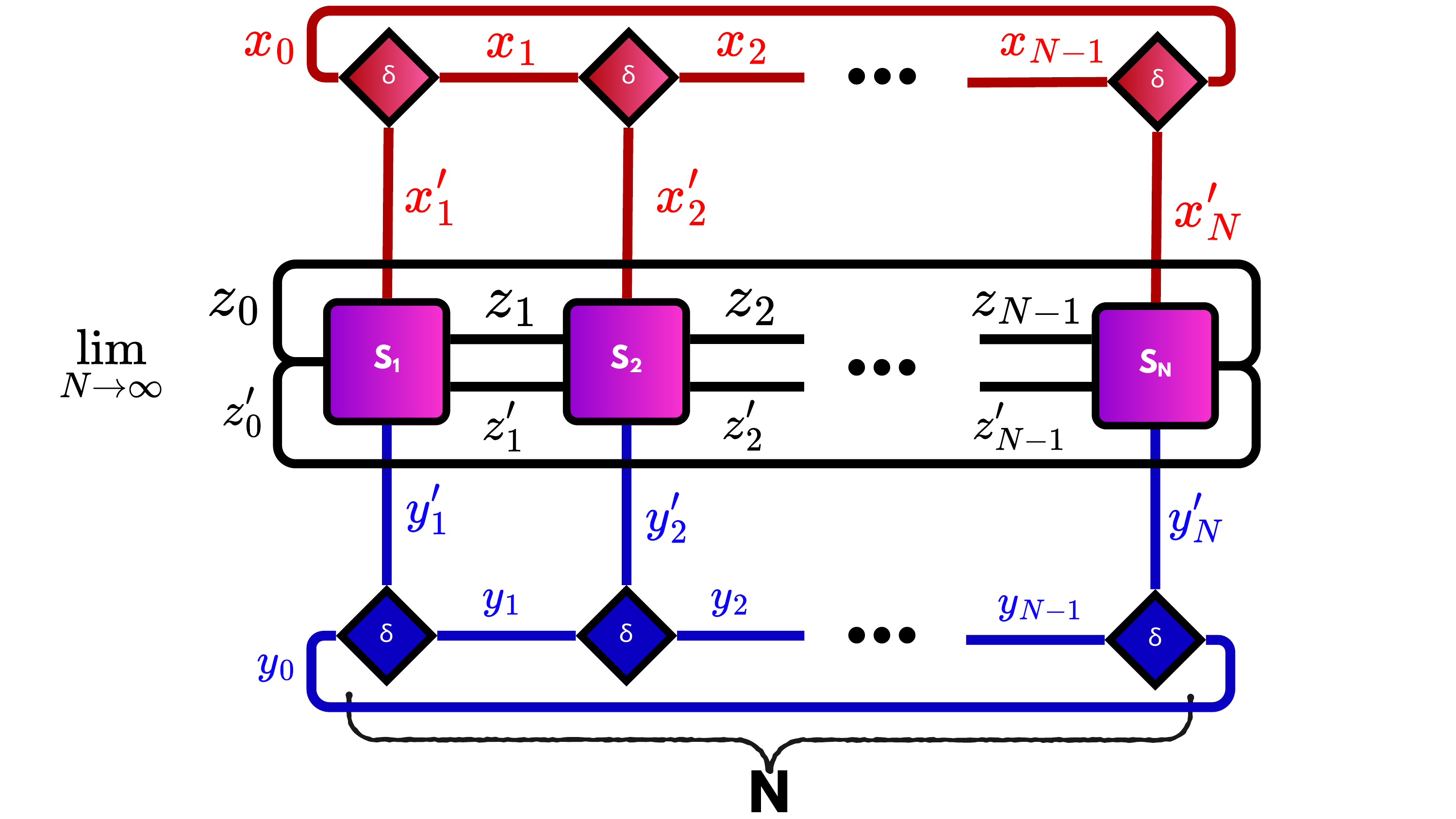}
    \caption{FTNILO of the Riemann Zeta Function zeros in donut shape. This FTN returns the amount of zeros of the Riemann Zeta Function in the $Re(s)>1$ region.}
    \label{fig: Riemann TN donut}
\end{figure*}

We define $a\%N = a \mod(N)$ to simplify the following equation. Due to the cyclic condition that we can impose with a modular operation, the FNTILO equation is
\begin{equation}\label{eq: Riemann 1 donut}
    \begin{gathered}
        \mathcal{N} =\lim_{N\rightarrow\infty} \int_{\mathbb{R^+}^{2N}}\int_{\mathbb{R}^{2N}}\int_{\mathbb{R}^{2N}}
        \prod_{n=1}^{N}\left(\delta(x_{n\%N}-x_{n-1})\delta(x'_{n}-x_{n-1})\delta(y_{n\%N}-y_{n-1})\delta(y'_{n}-y_{n-1})\times\right.\\
        \left.\times S_{n}(z_{n\%N},z_{n-1},x'_{n},y'_{n})\right)
        d\vec{x}d\vec{x'}d\vec{y}d\vec{y'}d\vec{z}d\vec{z'},
    \end{gathered}
\end{equation}
that developed is
\begin{equation}\label{eq: Riemann 1 developed donut}
    \begin{gathered}
        \mathcal{N} =\lim_{N\rightarrow\infty} \int_{\mathbb{R^+}^{2N}}\int_{\mathbb{R}^{2N}}\int_{\mathbb{R}^{2N}}
        \prod_{n=1}^{N}\left(\delta(x_{n\%N}-x_{n-1})\delta(x'_{n}-x_{n-1})\delta(y_{n\%N}-y_{n-1})\delta(y'_{n}-y_{n-1})\times\right.\\
        \left.\times \delta\left(z_{n\%N} - \left(z_{n-1}+Re\left(\frac{1}{n^{x_{n}'+iy_{n}'}}
        \right)\right)\right)
    \delta\left(z'_{n\%N} - \left(z'_{n-1}+Im\left(\frac{1}{n^{x_{n}'+iy_{n}'}}\right)\right)\right)\right) 
    d\vec{x}d\vec{x'}d\vec{y}d\vec{y'}d\vec{z}d\vec{z'}.
    \end{gathered}
\end{equation}
As we know, in the $Re(s)\geq1$ region there are no zeros of the Riemann Zeta function~\cite{Hadamard_1}, so this integral will be zero in the limit. However, it allows us to approach the interesting region of $1>Re(s)>0$, where all the interesting non-trivial zeros should be located.

\subsection{Riemann Function for $1>Re(s)>0$}
For $Re(s)>0$, the series is~\cite{Dirichlet}
\begin{equation}
    \zeta (s) = \frac{1}{s-1}\sum_{n=1}^{\infty}\left(\frac{n}{(n+1)^s}-\frac{n-s}{n^s}\right).
\end{equation}
If we consider only the $1>Re(s)>0$ region, we can neglect the factor $\frac{1}{s-1}$, so we will approach the function
\begin{equation}
    \zeta' (s) = \sum_{n=1}^{\infty}\left(\frac{n}{(n+1)^s}-\frac{n-s}{n^s}\right).
\end{equation}

Following the same structure that in the $Re(s)>1$ region, each $S$ operator computes one term of the series, and all of them collaborate to compute the global. This makes the FTN in Fig.~\ref{fig: Riemann TN} and Fig.~\ref{fig: Riemann TN donut} the ones we will use, but we change the $S$ functions. The current function is
\begin{equation}
    \begin{gathered}
    S_n(z_n,z_{n-1},x'_n,y'_n) = \delta\left(z_n - \left(z_{n-1}+Re\left(\frac{n}{(n+1)^{x_n'+iy_n'}}-\frac{n-x_n'-iy_n'}{n^{x_n'+iy_n'}}\right)\right)\right)\times\\
    \times\delta\left(z'_n - \left(z'_{n-1}+Im\left(\frac{n}{(n+1)^{x_n'+iy_n'}}-\frac{n-x_n'-iy_n'}{n^{x_n'+iy_n'}}\right)\right)\right).
    \end{gathered}
\end{equation}
Now we have to take into account the values of $s$ with $1>Re(s)>0$ and discard all $Re(s)=\frac{1}{2}$. We can do this by changing the $S$ operators or changing the integration region. We will consider only the integration region $(0,1)$ for the values $\vec{x}$ and $\vec{x}'$, and discard the value $\frac{1}{2}$ with a factor $1-\chi_{\frac{1}{2}}(x_n)$. This makes the global equation for the linear FTN
\begin{equation}\label{eq: Riemann 0}
    \begin{gathered}
        \mathcal{N} =\lim_{N\rightarrow\infty} \int_{(0,1)^{2N+1}}\int_{\mathbb{R}^{2N+1}}\int_{\mathbb{R}^{2(N+1)}}
        (1-\chi_{\frac{1}{2}}(x_0))
        \delta(z_0)\delta(z'_0)\times\\
        \times\prod_{n=1}^{N}\left(\delta(x_{n}-x_{n-1})\delta(x'_{n}-x_{n-1})\delta(y_{n}-y_{n-1})\delta(y'_{n}-y_{n-1})S_n(z_n,z_{n-1},x'_n,y'_n)\right)\times\\
        \times\delta(z_{N}) \delta(z'_{N})
        d\vec{x}d\vec{x'}d\vec{y}d\vec{y'}d\vec{z}d\vec{z'},
    \end{gathered}
\end{equation}
that developed is
\begin{equation}\label{eq: Riemann 0 developed}
    \begin{gathered}
        \mathcal{N} =\lim_{N\rightarrow\infty} \int_{(0,1)^{2N+1}}\int_{\mathbb{R}^{2N+1}}\int_{\mathbb{R}^{2(N+1)}}
        (1-\chi_{\frac{1}{2}}(x_0))
        \delta(z_0)\delta(z'_0)\times\\
        \times\prod_{n=1}^{N}\left(
        \delta(x_{n}-x_{n-1})\delta(x'_{n}-x_{n-1})\delta(y_{n}-y_{n-1})\delta(y'_{n}-y_{n-1})\times\right.\\
        \left.\times\delta\left(z_n - \left(z_{n-1}+Re\left(\frac{n}{(n+1)^{x_n'+iy_n'}}-\frac{n-x_n'-iy_n'}{n^{x_n'+iy_n'}}\right)\right)\right)\right.\times\\
    \left.\times\delta\left(z'_n - \left(z'_{n-1}+Im\left(\frac{n}{(n+1)^{x_n'+iy_n'}}-\frac{n-x_n'-iy_n'}{n^{x_n'+iy_n'}}\right)\right)\right)\right)\times\\
        \times\delta(z_{N}) \delta(z'_{N})
        d\vec{x}d\vec{x'}d\vec{y}d\vec{y'}d\vec{z}d\vec{z'}.
    \end{gathered}
\end{equation}

In the donut case, the equation is
\begin{equation}\label{eq: Riemann 0 donut}
    \begin{gathered}
        \mathcal{N} =\lim_{N\rightarrow\infty} \int_{(0,1)^{2N}}\int_{\mathbb{R}^{2N}}\int_{\mathbb{R}^{2N}}
        \prod_{n=1}^{N}\left(
        (1-\chi_{\frac{1}{2}}(x_{n\%N}))
        \delta(x_{n\%N}-x_{n-1})\delta(x'_{n}-x_{n-1})\delta(y_{n\%N}-y_{n-1})\delta(y'_{n}-y_{n-1})\times\right.\\
        \left.\times S_{n}(z_{n\%N},z_{n-1},x'_{n},y'_{n})\right)
        d\vec{x}d\vec{x'}d\vec{y}d\vec{y'}d\vec{z}d\vec{z'},
    \end{gathered}
\end{equation}
that developed is
\begin{equation}\label{eq: Riemann 0 developed donut}
    \begin{gathered}
        \mathcal{N} =\lim_{N\rightarrow\infty} \int_{(0,1)^{2N}}\int_{\mathbb{R}^{2N}}\int_{\mathbb{R}^{2N}}
        \prod_{n=1}^{N}\left(
        (1-\chi_{\frac{1}{2}}(x_{n\%N}))
        \delta(x_{n\%N}-x_{n-1})\delta(x'_{n}-x_{n-1})\delta(y_{n\%N}-y_{n-1})\delta(y'_{n}-y_{n-1})\times\right.\\
        \left.\times\delta\left(z_{n\%N} - \left(z_{n-1}+Re\left(\frac{n}{(n+1)^{x_n'+iy_n'}}-\frac{n-x_n'-iy_n'}{n^{x_n'+iy_n'}}\right)\right)\right)\right.\times\\
    \left.\times\delta\left(z'_{n\%N} - \left(z'_{n-1}+Im\left(\frac{n}{(n+1)^{x_n'+iy_n'}}-\frac{n-x_n'-iy_n'}{n^{x_n'+iy_n'}}\right)\right)\right)\right)\times\\
        \times d\vec{x}d\vec{x'}d\vec{y}d\vec{y'}d\vec{z}d\vec{z'}.
    \end{gathered}
\end{equation}

Equations \eqref{eq: Riemann 0 developed} and \eqref{eq: Riemann 0 developed donut} are called the \textit{Riemann FTNILO equations} and they indicate the number of non-trivial zeros of the Riemann Zeta Function outside the critical line.

\subsection{The answer of the Riemann hypothesis}
With the given equations, we can get the following. 

\begin{definition}[Delta consistency]
    Given a set of delta functions $\{\delta(f_i(\vec{x})), i\in[0,M-1]\}$, we say they are consistent if
    \begin{equation}
        \prod_{i=0}^{M-1} \delta(f_i(\vec{x})) \neq 0\ \forall \rho(\vec{x}),
    \end{equation}
    being $\rho(\vec{x})$ the set of variables non-shared between the deltas.
\end{definition}

\begin{theorem}[Riemann delta consistency]
    The Riemann hypothesis is false if and only if all its deltas in the Riemann FTNILO equation are consistent. This means that if some subset of the deltas is not consistent, the Riemann hypothesis is true.
\end{theorem}
\begin{proof}
    The FTNILO equations return the number of zeros in the region $1>Re(s)>0$ without the line $Re(s)=1/2$. This means that if the Riemann hypothesis is true, this number should be zero; otherwise, it will be another number because this is the only region that could have zeros. The equations consist only of Dirac deltas (the prefactor can be neglected changing the integration region). The function inside the integral is strictly positive, so we should go to the integrated function to determine if its value is zero or not. The only way the integral is zero is for the function to be zero, but the function is a product of functions, so it would be zero only if the product of a subset of its component functions is zero for all values of their non-common variables. This is precisely the Delta consistency definition. If all deltas are consistent in the limit, the integral must be equal to at least one.
\end{proof}

This result implies that the Riemann hypothesis problem, which originally would be proved by `checking all possible values of $s$' can be proven by finding only one Delta inconsistency. This search is outside the scope of this work, but we can give some possible approaches.

\begin{figure*}
    \centering
    \includegraphics[width=0.3\linewidth]{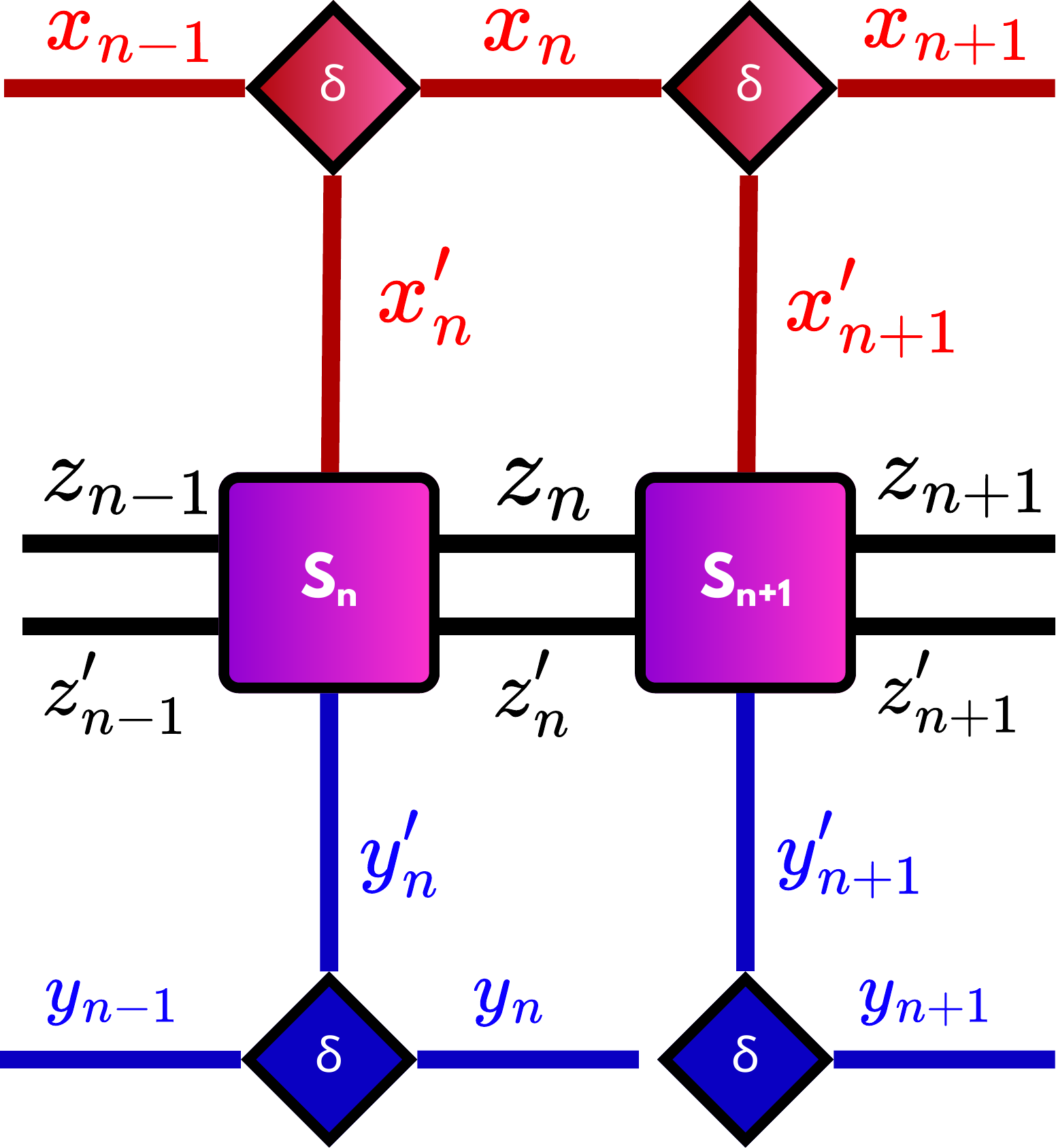}
    \caption{Donut Slice for the Riemann Zeta Function.}
    \label{fig: Riemann Donut slice}
\end{figure*}

The first is the real motivation for the donut version of the FTN. We could `cut' a slice of the donut, as in Fig.~\ref{fig: Riemann Donut slice}, and evaluate its consistency taking the limit. However, simple intuition tells us that the inconsistency must appear considering all the functions inside the integral. If we do not take into account one of the $S$ operators, the remaining term of the summation could be whatever. Similarly, if we do not consider one of the operators that sends the $s$ value, the associated $S$ operator could do the summation with other values of $s$. This means that the inconsistency of the product must be a global property instead of a local one.

However, another possible approach can be decomposing all the Dirac deltas into other functions, for example, Fourier transforms, and evaluating their `consistency'. In the same way, we can consider the limit version of the deltas and operate the functions that represent them before taking the limit.

A final possible approach is to make the integral in some way. To demonstrate that it is false, we only have to locate an infinitesimal point in the integration region with a non-zero value, which is equivalent to finding a zero.


\newpage
\section{Conclusions}
We have presented a new formalism to create equations to extract properties from functions. We have applied it to the minimization and inversion of functions and for obtaining the number and positions of certain values of them. We have also applied it to make some statements about cryptography. Finally, we have applied it to determine an explicit equation that returns the solution of the Riemann hypothesis. We also showed that its demonstration requires only a check of consistency over the Riemann FTNILO equations. Additionally, we have explored its connection with the discrete formalism MeLoCoToN, for combinatorial problems, and their combination for hybrid problems.

Further lines of research could be the generalization of the formalism to approach other problems, such as the resolution of differential or integral equations, the application of specific functions, such as cryptographic ones, and the modeling of abstract mathematical problems. For example, other Millennium Prize Problems could be studied from this framework.

An interesting question is the computation of these explicit equations. As we can see, these equations are hard to compute with classical resources, sometimes impossible, even with the transformations proposed. However, they could describe physical systems with related measurable quantities corresponding to these equations. If these physical systems exist, we could compute these equations efficiently, allowing the solution of certain types of problems. The study of their possible existence could be an interesting line of research.

\section*{Acknowledgment}
This work has been developed in the `When Physics Becomes Science' project of \href{https://www.youtube.com/@whenphysics}{When Physics}, an initiative to recover the original vision of science.

\bibliographystyle{unsrtnat}
\bibliography{references}

\end{document}